\documentclass[12pt,a4paper,fleqn]{article}
\pdfoutput=1
\usepackage[english]{babel}
\RequirePackage{ucs}
\RequirePackage[utf8]{inputenc}
\RequirePackage[T1]{fontenc}
\RequirePackage{amsthm,amsmath}
\RequirePackage{amssymb,amstext}
\RequirePackage{amsfonts,stmaryrd}
\RequirePackage{amsbsy} 
\RequirePackage{mathrsfs}
\RequirePackage{bm}
\RequirePackage{bbm}
\RequirePackage{aliascnt,ifthen}
\RequirePackage{graphicx}
\RequirePackage{color,calc}
\RequirePackage{times}
\RequirePackage{booktabs}
\RequirePackage{dsfont} 
\usepackage[margin = 1in]{geometry}
\usepackage{url}
\usepackage{verbatim}
\usepackage{subcaption}
\usepackage[pdfpagemode=UseOutlines ,plainpages=false,
hyperindex=true,colorlinks=true]{hyperref}
	\makeatletter
	\Hy@breaklinkstrue
	\makeatother

\definecolor{darkred}{rgb}{0.6,0,0.1}
\definecolor{darkgreen}{rgb}{0,0.5,0}
\definecolor{darkblue}{rgb}{0,0,0.5}

\hypersetup{colorlinks ,linkcolor=gray ,filecolor=darkgreen, urlcolor=darkblue ,citecolor=black}

\RequirePackage{paralist}
\RequirePackage{relsize}
\usepackage{enumitem}

\usepackage{natbib}
\usepackage{hypernat}
\renewcommand{\cite}{\citet}
\usepackage{abrevpackage}

\definecolor{dgreen}{rgb}{0,0.5,0}
\definecolor{dblue}{rgb}{0,0,0.5}
\definecolor{dred}{rgb}{0.6,0.0,0.1}
\definecolor{dgold}{rgb}{0.5,0.3,0.0}
\definecolor{dvio}{rgb}{0.6,0.3,0.5}
\definecolor{gray}{rgb}{0.5,0.5,0.5}
\definecolor{dbraun}{rgb}{.5,0.2,0}

\newcommand{\dgrau}{\color{gray}}

\newcommand{\colre}{dred}
\newcommand{\colas}{dblue}
\newcommand{\colrem}{dgold}
\newcommand{\colil}{dgreen}

\usepackage{setspace}
\usepackage{sectsty}
\allsectionsfont{\normalsize}
\newtheoremstyle{styre}% name
  {1.1\topsep}%      Space above
  {\topsep}%      Space below
  {\normalfont\itshape}%         Body font
  {}%         Indent amount (empty = no indent, \parindent = para indent)
  {\color{\colre}}% Thm head font
  {.}%        Punctuation after thm head
  {.5em}%     Space after thm head: " " = normal interword space;
  {\thmname{\textbf{#1}\xspace}{\xspace\dgrau\thmnumber{#2}}\thmnote{\xspace\textit{\small(#3)}}}%         Thm head spec (can be left empty, meaning `normal')

\newtheoremstyle{styas}% name
  {1.1\topsep}%      Space above
  {\topsep}%      Space below
  {\normalfont\itshape}%         Body font
  {}%         Indent amount (empty = no indent, \parindent = para indent)
  {\color{\colas}}% Thm head font
  {.}%        Punctuation after thm head
  {.5em}%     Space after thm head: " " = normal interword space;
  {}%         Thm head spec (can be left empty, meaning `normal')

\newtheoremstyle{styrem}% name
  {1.1\topsep}%      Space above
  {\topsep}%      Space below
  {\normalfont\itshape}%         Body font
  {}%         Indent amount (empty = no indent, \parindent = para indent)
  {\color{\colrem}}% Thm head font
  {.}%        Punctuation after thm head
  {.5em}%     Space after thm head: " " = normal interword space;
  {\thmname{\textbf{#1}\xspace}{\xspace\dgrau\thmnumber{#2}}\thmnote{\xspace\textit{\small(#3)}}}%{}%         Thm head spec (can be left empty, meaning `normal')

\newtheoremstyle{styil}% name
  {1.1\topsep}%      Space above
  {\topsep}%      Space below
  {\normalfont\rmfamily}%         Body font
  {}%         Indent amount (empty = no indent, \parindent = para indent)
  {\color{\colil}}% Thm head font
  {.}%        Punctuation after thm head
  {.5em}%     Space after thm head: " " = normal interword space;
  {\thmname{\textbf{#1}\xspace}{\xspace\dgrau\thmnumber{#2}}\thmnote{\xspace\textit{\small(#3)}}}%         Thm head spec (can be left empty, meaning `normal')

\newtheoremstyle{stypro}%
	{0.5\topsep}%space above
	{1.1\topsep}%space below
	{\upshape}%		body font
	{}%				indent amount
	{}%	theorem head font
	{.}%			punctuation after theorem head
	{.5em}%			space after theorem head
	{\thmnote{\textit{#3}}}%         Thm head spec (can be left empty, meaning `normal')

\theoremstyle{styre}\newtheorem{pr}{Proposition}[section]
\newaliascnt{co}{pr}
\theoremstyle{styre}\newtheorem{co}[co]{Corollary}
\aliascntresetthe{co}
\newaliascnt{thm}{pr}
\theoremstyle{styre}
\aliascntresetthe{thm}
\newaliascnt{lem}{pr}
\theoremstyle{styre}\newtheorem{lem}[lem]{Lemma}
\aliascntresetthe{lem}
\newaliascnt{rem}{pr}
\theoremstyle{styrem}
\aliascntresetthe{rem}
\newaliascnt{il}{pr}
\theoremstyle{styil}\newtheorem{il}[il]{Illustration}
\aliascntresetthe{il}
\theoremstyle{styas}\newtheorem{ass}{Assumption}

\theoremstyle{styas}

\theoremstyle{stypro}

\usepackage{cleveref}
\crefname{pr}{\color{\colre}Proposition}{\color{\colre}Propositions}
\crefname{co}{\color{\colre}Corollary}{\color{\colre}Corollaries}
\crefname{thm}{\color{\colre}Theorem}{\color{\colre}Theorems}
\crefname{lem}{\color{\colre}Lemma}{\color{\colre}Lemmata}
\crefname{ass}{\color{\colas}Assumption}{\color{\colas}Assumptions}
\crefname{de}{\color{\colas}Definition}{\color{\colas}Definitions}
\crefname{rem}{\color{\colrem}Remark}{\color{\colrm}Remarks}
\crefname{il}{\color{\colil}Illustration}{\color{\colil}Illustrations}
\crefname{table}{Table}{Tables}
\crefname{section}{Section}{Sections}

\usepackage{environ}
\NewEnviron{te}{%
{\BODY}%
}

\numberwithin{equation}{section} 

\makeatletter
\newcommand{\mylabel}[2]{#2\def\@currentlabel{#2}\label{#1}}
\makeatother

\makeatletter% 
\def\@fnsymbol#1{\ensuremath{\ifcase#1\or * \or ** \or 2 \or 3 \or  *\or  \star \or 4\or  , \or 
g\or h\or i\else\@ctrerr\fi}}% 
\makeatother% 

\author{Jan Johannes$\;^{a}$ and Bianca Neubert$\;^{a*}$  \\ \small $\;^{a}$
Institut f\"{u}r Mathematik (IMa), Interdisciplinary Center for Scientific Computing (IWR),\\ \small Ruprecht-Karls-Universit\"{a}t Heidelberg, Germany; $\;^{*}$neubert@math.uni-heidelberg.de}
\date{} 
\title{\textbf{Goodness-of-fit testing from observations with multiplicative measurement error}} 
\begin{document} 
\renewcommand{\abstractname}{\vspace{-\baselineskip}}
% --------------------------------------------------------------------
% <<Abstract>>
% --------------------------------------------------------------------

\maketitle
 
\begin{abstract}
Given observations from a positive random variable contaminated by multiplicative measurement error, we consider a nonparametric goodness-of-fit testing task for its unknown density in a non-asymptotic framework. We propose a testing procedure based on estimating a quadratic functional of the Mellin transform of the unknown density and the null. We derive non-asymptotic testing radii and testing rates over Mellin–Sobolev spaces, which naturally characterize regularity and ill-posedness in this model. By employing a multiple testing procedure with Bonferroni correction, we obtain data-driven procedures and analyze their performance. Compared with the non-adaptive tests, their testing radii deteriorate by at most a logarithmic factor. We illustrate the testing procedures with a simulation study using various choices of densities.
\end{abstract} 
{\footnotesize
\begin{tabbing} 
\> Bonferroni aggregation\\[.2ex] 
\end{tabbing}}

% --------------------------------------------------------------------
% <<Content>>
% --------------------------------------------------------------------
\section{Introduction}
In this work, we consider nonparametric data-driven goodness-of-fit hypothesis testing for the density of a strictly positive random variable $X$ in a multiplicative measurement error model. More precisely, let $X$ and $U$ be strictly positive independent random variables admitting, respectively, an unknown density $\denX$ and a known density $\denU$ with respect to the Lebesgue measure $\pLm$ on the positive real line $\pRz$.  We have access to an independent and identically distributed (i.i.d.) sample of size $n\in\Nz$ from $Y = X U$  admitting a density $\denY$ given for $y\in\pRz$  by 
 \begin{align}\label{eq::multiplicative-convolution2}
 	\denY (y) = (\denX \mc \denU)(y):= \int_{\pRz}  \denX(x) \denU(y/x)x^{-1}d\pLm (x).
 \end{align}
For a prescribed density $\denNX$ we  propose a data-driven testing procedure based only on  observations  of $Y$ for the goodness-fit testing task
 \begin{align*}
    	H_0\colon \denX = \denNX  \quad\text{ against }\quad H_1 \colon \denX \not= \denNX,
 \end{align*}
which naturally leads to a statistical inverse problem called multiplicative deconvolution.

\cite{Vardi1989} and \cite{VardiZhang1992} introduce and thoroughly investigate multiplicative censoring, which corresponds to the specific multiplicative deconvolution model where the multiplicative error $U$ is uniformly distributed on $[0,1]$. As explained and motivated in \cite{VANES2000295}, multiplicative censoring represents a common challenge in survival analysis. 
There is a substantial body of work on estimation of the unknown density $\denX$. Under multiplicative censoring, for example,  \cite{Andersen2001} employ series expansion methods; \cite{brunel2016} propose a kernel estimator; \cite{Comte2016} develop a projection estimator based on the Laguerre basis; and \cite{Belomestny2016} study the case of a Beta-distributed error. All of these approaches fall within the broader framework of the multiplicative measurement error model.
Nonparametric density estimation in this setting has been investigated by \cite{Brenner2021}, who assume a known error density and apply spectral cut-off regularization, and by \cite{Siebel2023}, who address the case of an unknown error density. Further contributions include  \cite{Miguel2021}, who consider estimation of a linear functional of the unknown density $\denX$, and  \cite{Belomnestny2020}, who analyze pointwise density estimation. Pointwise estimation of the cumulative distribution function, as studied in \cite{belomestny2024}, also covers scenarios where absolute continuity does not hold.
The estimation of a quadratic functional of the unknown density has been studied in \cite{Comte2025}. The central tool for analysing multiplicative deconvolution is the multiplication theorem. For a density $\denY = \denX *\denU$ and the corresponding Mellin transformations $\Mcg{c}{\denY}, \Mcg{c}{\denX}, \Mcg{c}{\denU}$ (defined below) the theorem states that
$\Mcg{c}{\denY} =\Mcg{c}{\denX} \Mcg{c}{\denU}$. This property is exploited in \cite{Belomnestny2020}  to construct a kernel density estimator of $\denX$ while \cite{Miguel2021} and \cite{Comte2025} employ it to develop a plug-in estimator.

We now turn to the testing task.  Asymptotic rates of testing for nonparametric alternatives is introduced in the series of papers by \cite{MR1257983,MR1259685,MR1257978}. For a fixed sample size, a nonasymptotic radius of testing has been studied by \cite{zbMATH01870137}, \cite{Laurent_2012}, \cite{Collier_2017} and \cite{Marteau_2017} amongst others.  Its connection to the asymptotic approach is studied in \cite{MarteauSapatinas2015}. 
The link between quadratic functional estimation and testing is amongst others explored in  \cite{Collier_2017}, \cite{Kroll2019} and  \cite{SchluttenhoferJohannes2020b,SchluttenhoferJohannes2020}. Estimation of quadratic functionals of densities have extensively been treated in the literature by \cite{LaurentMassart2000}, \cite{Laurent2005} and \cite{BUTUCEA201131}, to name but a few. In particular we refer to  \cite{Butucea2007}, \cite{SchluttenhoferJohannes2020b} and \cite{Comte2025} for measurement models with additive real, circular and multiplicative errors, respectively. Goodness-of-fit testing for additive real measurement error models has been studied for example  by \cite{Butucea2007} while \cite{SchluttenhoferJohannes2020} consider circular errors. In this paper, we study the multiplicative case. We propose a test statistic by replacing the unknown Mellin transform with its empirical counterpart, borrowing ideas from \cite{Comte2025}.  More precisely, to distinguish between the null hypothesis and the alternative, we consider a weighted $\Lp[2]{}$-distance between the Mellin transforms of $\denX$ and $\denNX$   allowing to study goodness-of-fit hypotheses on the density $\denX$, its derivatives or the associated survival function. Considering this general framework, we derive upper bounds for the radius of testing.

The accuracy of the proposed test statistic relies crucially on an optimal choice of a tuning parameter which is unknown in practice.  
In the literature, data-driven testing strategies have been studied in both asymptotic and non-asymptotic frameworks (e.g., \cite{Spokoiny_1996},  \cite{Baraud_2003}, and \cite{Fromont_2006}). In particular, we refer to \cite{ButuceaMatiasPouet2009} and \cite{Butucea2007} for additive real measurement error models, and to \cite{SchluttenhoferJohannes2020} for the circular case. Our aim is to derive in a multiplicative measurement error model upper bounds for the radius of testing of a fully data-driven procedure using a Bonferroni approach. 

The paper is organized as follows.
We first review the Mellin transform, present, the testing task, and propose  the test statistic in \cref{sec::dn}. In \cref{sec::bounds-quantiles}, we derive bounds for the quantiles of the proposed test statistic. In \cref{sec::uperbound-radius-testing}, we establish upper bounds for the radius of testing. In \cref{sec::radius-max-test}, we propose a data-driven testing procedure based on the max-test and also provide upper bounds for its radius of testing. Finally, in \cref{sec::sim-testing}, we illustrate the theoretical results with a small simulation study.
\section{Methodology}\label{sec::dn}
Before stating the testing task and defining a test statistic, we introduce the Mellin transform and recall the multiplicative deconvolution problem.
\paragraph{Mellin transform}
Consider on the Borel-measurable space $(\pRz, \psB)$ the restriction $\pLm$ of the Lebesgue-measure $\Lm$ on $\pRz$. For any $c\in\Rz$ denote by  $\basMSy{c}\pLm$ the $\sigma$-finite measure with Lebesgue-density $\basMSy{c}$ given by $x\mapsto x^c$. Introduce for $s\in\Rz_{\geq1}$ the set $\Lp[s]{}(\basMSy{c}) := \Lp[s]{+}(\pRz, \sB^+, \basMSy{c}\pLm)$ of all complex-valued $\Lp[s]{+}(\basMSy{c}) $-integrable functions. Analogously, denote by $\Lp[2]{} := \Lp[2]{}(\Rz, \sB, \Lm)$ the set of all complex-valued square-Lebesgue-integrable functions. Then,  the Mellin transform is the operator ${\Mg{c}\in \mathbb{L}(\Lp[2]{+}(\basMSy{2c-1}), \Lp[2]{})}$   satisfying for $g\in \Lp[1]{+}(\basMSy{c-1})\cap \Lp[2]{+}(\basMSy{2c-1})$ and  $t\in\Rz$
\begin{align*}
	\Mcg{c}{g}(t) = \int_{\pRz} x^{c-1+  2\pi\iota t} g(x) \pLm(dx) = \int_{\pRz} x^{2 \pi \iota t} g(x) (\basMSy{c-1}\pLm)(x).
\end{align*} 
Moreover,  it fulfils  $\langle h_1, h_2 \rangle_{\Lp[2]{+}(\basMSy{2c-1})} = \langle \Mcg{c}{h_1}, \Mcg{c}{h_2} \rangle_{\Lp[2]{}}$ for all $h_1,h_2\in\Lp[2]{+}(\basMSy{2c-1})$, i.e. $\Mg{c}$ is unitary, and, in particular,  the \textit{Plancherel equality} holds, that is,
\begin{align}\label{eq::Plancherel}
	\Vert h_1 \Vert^2_{\Lp[2]{+}(\basMSy{2c-1})} = \Vert \Mcg{c}{h_1} \Vert_{\Lp[2]{}}^2.
\end{align}
The adjoint $\Mgi{c}\in \mathbb{L}(\Lp[2]{}, \Lp[2]{+}(\basMSy{2c-1}))$ of the Mellin transform satisfies for $G\in\Lp[1]{}\cap \Lp[2]{}$ and $x\in\pRz$
\begin{align*}
	\Mcgi{c}{G} (x) = \int_\Rz x^{-c- 2\pi \iota t} G(t) d\Lm(t).
\end{align*}

In analogy to the additive convolution theorem, see \cite{Meister2009}, there is a \textit{multiplicative convolution theorem}. It states that for  $h_1, h_2\in\Lp[1]{+}(\basMSy{c-1})$
\begin{align}
	\Mcg{c}{h_1 * h_2} = \Mcg{c}{h_1} \cdot \Mcg{c}{h_2}\label{eq::convolution-theorem}
\end{align}
where $*$ denotes the multiplicative convolution defined in \cref{eq::multiplicative-convolution2}. We refer to \cite{BrennerMiguelDiss}  for a detailed discussion of the Mellin transform.

\paragraph{Multiplicative deconvolution}
In the multiplicative measurement error model, we are interested in an unknown density $\denX\colon\pRz \rightarrow \Rz_{\geq 0}$ of a positive random variable $X$ given independent and identically distributed copies $(Y_j)_{j\in\nset{n}}$ of
\begin{align}\label{eq::multi-model}
	Y = XU
\end{align}
where $X$ and $U$ are independent of each other and $U$ has a known density $\denU\colon\pRz \rightarrow \Rz_{\geq 0}$. In this setting the density $\denY\colon\pRz \rightarrow \Rz_{\geq 0}$ of $Y$ is given by $\denY = \denX * \denU.$
Making use of the multiplicative convolution theorem \cref{eq::convolution-theorem} it holds that $\Mcg{c}{\denY} = \Mcg{c}{\denX} \cdot \Mcg{c}{\denU}$. 
\paragraph{Goodness-of-fit testing task}

Throughout this paper, we assume that the random variables admit densities with respect to the Lebesgue measure on $\pRz$ and that the corresponding Mellin transforms exist. These assumptions are summarized in the following set of functions:
\begin{align*}
	\regD := \{h\in\Lp[1]{+}(\basMSy{c-1})\cap \Lp[2]{+}(\basMSy{2c-1}) : h \text{ Lebesgue density on }\pRz \}.
\end{align*}
Here and subsequently, the  quantities of the following assumption are fixed and assumed to be known.

\begin{ass}\label{ass:well-definedness-testing}
	Consider the multiplicative measurement error model, an arbitrary measurable symmetric density function $\wFSy\colon\Rz \rightarrow \Rz_{\geq 0}$ and $c\in\Rz$. In addition, let $\denNX,\denU\in\regD$.
\end{ass}
Under \cref{ass:well-definedness-testing} we are interested in the goodness-of-fit testing task
\begin{align}\label{eq::testingtask-simple2}
	H_0\colon \denX = \denNX  \quad\text{ against }\quad H_1 \colon \denX \not= \denNX.
\end{align}
The performance of a test is measured by how well it is able to distinguish between the null hypothesis and elements that are in some sense separated from the null. For this, first define for an arbitrary measurable and symmetric density function $\text{u}\colon\Rz\rightarrow\Rz$ the class 
\begin{align*}
	\regFw{\text{u}} :=  \{ h\in\Lp[2]{+}(\basMSy{2c-1})\cap \Lp[1]{+}(\basMSy{c-1}) :  \Mcg{c}{h}\in\Lp[2]{}(\text{u}^2) \}.
\end{align*}
Now, we consider the following norm to measure the separation of densities. 
Under \cref{ass:well-definedness-testing}, we define for $h\in\regFw{\wF}$ analogously to \cite{Comte2025} the quadratic functional 
\begin{align}\label{eq::norm-def-2}
	\wqf(h) := \Vert \Mcg{c}{h} \Vert_{\Lp[2]{}(\iwF[]{2})}^2 = \int_{\Rz}  \vert \Mcg{c}{h} (t)\vert^2 \iwF[]{2} (t) d\Lm(t).
\end{align}

We define an \textit{energy set} for the separation radius $\rho\in\pRz$ as
\begin{align*}
	\regrho := \left\{h\in\regFw{\wF}: \wqf(h) \geq \rho^2 \right\}
\end{align*}
and write shortly
\begin{align*}
	\regD\cap (\denNX +  \regrho)  := \{h\in\regD:  h-\denNX \in \regrho\}.
\end{align*}
Consequently, under \cref{ass:well-definedness-testing} the testing task can be written as
\begin{align*}
	H_0\colon \denX = \denNX  \quad\text{ against }\quad H_1 \colon \denX \in\regD\cap (\denNX + \regrho).
\end{align*}

\begin{il}\label{il::example-parameters}
	Different choices of the function $\wF$ allows us to cover the following possible goodness-of-fit testing tasks.
	\begin{itemize}
		\item[(i)] If $\wFSy = 1$, we get using the Plancherel type identity \cref{eq::Plancherel} that
		\begin{align*}
			\wqf(\denX) = \Vert \Mcg{c}{\denX} \Vert_{\Lp[2]{}}^2 = \Vert \denX \Vert_{\Lp[2]{+}(\basMSy{2c-1})}^2 = \int_{\pRz} \vert f(x)\vert^2 x^{2c-1} dx
		\end{align*}
		and  $\wqf(f) = \Vert \denX \Vert_{\Lp[2]{+}}^2$ in the special case of $c=\frac12$. Consequently, we cover the testing task 
		\begin{align*}
			H_0\colon \denX = \denNX  \quad\text{ against }\quad H_1 \colon \denX \in\regD,  \Vert \denX - \denNX\Vert_{\Lp[2]{+}}^2 \geq \rho.
		\end{align*}
		\item[(ii)] If $\iwF{2}(t)= \frac{1}{(c-1)^2+4\pi^2 t^2}$ for all $t\in\Rz$, we consider the quadratic functional evaluated at the survival function  $S$ of $X$, i.e.
		\begin{align*}
			\wqf (f)= \Vert S \Vert_{\Lp[2]{+}(\basMSy{2c-1})} =  \int_{\pRz} \vert S(x)\vert^2 x^{2c-1} dx.
		\end{align*}
		and  $\wqf(f) = \Vert S \Vert_{\Lp[2]{+}}^2$ in the special case of $c=\frac12$. For more details see \cite{BrennerMiguel-Phandoidaen2023}.  Consequently, we cover the testing task 
		\begin{align*}
			H_0\colon S = S_o \quad\text{ against }\quad H_1 \colon \denX \in\regD,  \Vert S - S_o \Vert_{\Lp[2]{+}}^2 \geq \rho.
		\end{align*}
		\item[(iii)] Let $\denX\in\mathcal{C}^\infty_0(\pRz)$, i.e. smooth with compact support in $\pRz$ such that for any $\beta\in\Nz$ the derivatives $D^\beta[\denX]:= \frac{d^\beta}{dx^\beta}f$ exist. 
		If $\iwF{2}(t)= \prod_{j=1}^\beta ((c+\beta-j)^2+4\pi^2t^2)$ for  $\beta\in\Nz$ then we have that 
		\begin{align*}
			\wqf (f)= \Vert D^\beta[\denX] \Vert_{\Lp[2]{+}(\basMSy{2(c+\beta)-1})}^2 =  \int_{\pRz} \vert D^\beta[\denX](t)\vert^2 x^{2(c+\beta)-1} dx,
		\end{align*}
		see Proposition 2.3.11, \cite{BrennerMiguelDiss}, and  $\wqf (f) = \Vert D^\beta[\denX] \Vert_{\Lp[2]{+}}^2$ in the special case $c=\frac12 - \beta$.
		Consequently, we cover the testing task 
		\begin{align*}
			H_0\colon  D^\beta[\denX] = D^\beta[\denNX] \quad\text{ against }\quad H_1 \colon \denX \in\regD,  \Vert  D^\beta[\denX] - D^\beta[\denNX]\Vert_{\Lp[2]{+}}^2 \geq \rho.
		\end{align*}
		\end{itemize}
\end{il}

Further, we assume some kind of regularity for the densities under the alternative. Analogously to the quadratic functional estimation in \cite{Comte2025} consider the following assumption.
\begin{ass}\label{ass::classes-testing}
\begin{enumerate}
	\item[(i)] Let $\operatorname{s}\colon\Rz\rightarrow \Rz_{\geq 1}$ be a symmetric, non-decreasing regularity function such that $\operatorname{s}(t)\rightarrow \infty$ as $\vert t\vert \rightarrow \infty$. 
	\item[(ii)] Assume that the symmetric function $\wF /\operatorname{s}$ is  non-increasing such that $\wF(t)/\operatorname{s}(t) = o(1)$ as $\vert t\vert \rightarrow \infty$.
\end{enumerate}
\end{ass}
Then, for $R\in\pRz$ define
\begin{align}\label{eq::reg-class}
		\regC := \left\{h\in\regFw{\operatorname{s}}:  \left\Vert\operatorname{s} \Mcg{c}{h}\right\Vert_{\Lp[2]{}}^2 \leq R^2 \right\}
\end{align}
and write shortly
\begin{align*}
	\regD\cap (\denNX + \regC \cap \regrho)  := \{h\in\regD:  h-\denNX \in \regC \cap \regrho\}.
\end{align*}
Consequently, under \cref{ass:well-definedness-testing,ass::classes-testing} we analyse the testing task 
\begin{align}\label{eq::testingtask-in-f}
	H_0\colon \denX = \denNX  \quad\text{ against }\quad H_1 \colon \denX \in\regD\cap (\denNX + \regC \cap \regrho).
\end{align}
Roughly speaking, in minimax testing one searches for the smallest $\rho$ such that \cref{eq::testingtask-in-f} is still testable with small error probabilities. We measure the accuracy of a test $\Delta\colon\pRz^n\rightarrow \{0,1\}$ by its maximal risk defined as the sum of the type I error probability and the maximal type II error probability over the $\rho$-separated alternative
\begin{align*} 
	\mathcal{R} \left( \Delta \vert \regD, \regC , \rho^2 \right) &= \ipM[\denU,\denNX]{n}\left( \Delta= 1 \right) + \sup_{\denX \in\regD\cap (\denNX + \regC \cap \regrho)} \ipM[\denU,\denX]{n}\left( \Delta = 0\right).
\end{align*}
A widely used approach consists in deriving a test statistic $\hat{q}^2$ which estimates $\wqf(\denX-\denNX)$ and forming a test given by 
\begin{align*}
	\Delta:= \mathds{1}_{\{\hat{q}^2 \geq \tau\}}
\end{align*}
for some suitable value $\tau\in\pRz$. In this paper, the test statistic is inspired by the estimator proposed in \cite{Comte2025}.
Before proposing an estimator of $\wqf(\denX-\denNX)$, we introduce some notations and an additional assumption. 
For brevity, we write for $x\in\pRz$ and $t\in\Rz$
\begin{align*}
	\Xtilde{x}{t} := x^{c-1 + 2\pi i t} \quad \text{ and }\quad \cfun{x}{t}:= \Xtilde{x}{t} -  \Ex[\denY][\Xtilde{Y}{t}],
\end{align*}
where evidently $ \Ex[\denY][\Xtilde{Y}{t}] = \Mcg{c}{\denY}(t)$ and $\cfun{Y}{t}$ is centred.

\begin{ass}\label{ass::error-well-defined}
	 $\Mcg{c}{\denU}(t)\not=0$ for all $t\in\Rz$ and $\mathds{1}_{[-k,k]}/\Mcg{c}{\denU}\in\Lp[\infty]{+}(\wF)$ for all $k\in\pRz$. 
\end{ass}
We first see that under \cref{ass:well-definedness-testing,ass::error-well-defined}  it holds that for any $\denX\in\regD$
\begin{align}
	\wqfk{k}(\denX-\denNX) &= \Vert \indc{-k}{k} ({\Mcg{c}{\denX}} - \Mcg{c}{\denNX} )\Vert^2_{\Lp[2]{}(\iwF{2})} \nonumber\\ 
	&= \Vert \indc{-k}{k} {\Mcg{c}{\denX}}  \Vert^2_{\Lp[2]{}(\iwF{2})} - 2 \langle \indc{-k}{k} {\Mcg{c}{\denX}} , \indc{-k}{k}\Mcg{c}{\denNX} \rangle_{\Lp[2]{}(\iwF{2})}\nonumber\\
	&\qquad + \Vert \indc{-k}{k} {\Mcg{c}{\denNX}}  \Vert^2_{\Lp[2]{}(\iwF{2})}\nonumber
\end{align}
and with the convolution theorem (c.f. \cref{eq::convolution-theorem}) we get
\begin{align}
	\wqfk{k}(\denX-\denNX)& = \Vert \indc{-k}{k} {\Mcg{c}{\denY}}/ \Mcg{c}{\denU}  \Vert^2_{\Lp[2]{}(\iwF{2})}\nonumber\\
	&\qquad - 2 \langle \indc{-k}{k} {\Mcg{c}{\denY}}/ \Mcg{c}{\denU}  , \indc{-k}{k}\Mcg{c}{\denNY}/ \Mcg{c}{\denU}  \rangle_{\Lp[2]{}(\iwF{2})}\nonumber\\
	&\qquad + \Vert \indc{-k}{k} {\Mcg{c}{\denNY}}/ \Mcg{c}{\denU}   \Vert^2_{\Lp[2]{}(\iwF{2})}.\label{eq::quadraticfunctionalreg2}
\end{align}
Since $\denNX$ and $\denU$ are known, the third term of \cref{eq::quadraticfunctionalreg2} is known as well. For the second term, we introduce the following estimator:
\begin{align*}
	\widehat{S}_{k} &:= \langle \indc{-k}{k} \Xtilde{Y_j}{\cdot}/ \Mcg{c}{\denU}  , \indc{-k}{k}\Mcg{c}{\denNY}/ \Mcg{c}{\denU}  \rangle_{\Lp[2]{}(\iwF{2})}\\
	&=\frac{1}{n} \sum_{j \in\llbracket n \rrbracket}  \int_\Rz \indc{-k}{k}(t) \vert\Mcg{c}{\denU}(t)\vert^{-2} \Xtilde{Y_j}{t} \overline{\Mcg{c}{\denNY}(t)}  \iwF{2}(t) d\Lm(t).
\end{align*}
The first term of \cref{eq::quadraticfunctionalreg2} we estimate by
\begin{align}
	\widehat{T}_k :=   \frac{1}{n(n-1)}\sum_{\substack{j\not= l\\ j,l \in\llbracket n\rrbracket }}  \int_{-k}^k \frac{\Xtilde{Y_j}{t}\Xtilde{Y_l}{-t}}{\vert\Mcg{c}{\denU}(t)\vert^2} \iwF[]{2} (t)  d\Lm(t). \label{eq::quadraticestim-testing}
\end{align}
Due to the independence of $(Y_j)_{j\in\llbracket n\rrbracket }$ we have that
\begin{align*}
	\wqfk{k}(\denX) &:= \iEx[\denY]{n} [\widehat{T}_k] = \int_{-k}^k \frac{\Ex[\denY][\Xtilde{Y_1}{t}] \Ex[\denY][\Xtilde{Y_2}{-t}]}{\vert\Mcg{c}{\denU}(t)\vert^2} \iwF[]{2} (t)  d\Lm(t) \\
	&= \int_{-k}^k  \vert \Mcg{c}{\denX}(t)\vert^2  \iwF[]{2} (t)  d\Lm(t).\nonumber
\end{align*}
Consequently, combining all three terms, the test statistic defined by
\begin{align}\label{eq::estimatorq2kb}
	\eqfk{k}  := 	\widehat{T}_k  - 2	\widehat{S}_{k} + \wqfk{k}(\denNX).
\end{align}
gives an unbiased estimator of $\wqfk{k}(\denX-\denNX) $. Next, we derive bounds for the  quantiles of the proposed test statistic, which are subsequently used to establish upper bounds of the radius of testing.
\section{Bounds for the quantiles of the test statistic}\label{sec::bounds-quantiles}
In this section, we derive bounds for the quantiles of the test statistic $\eqfk{k}$ defined in \cref{eq::estimatorq2kb}.
To define the critical value $\tau_k(\alpha)$ for $\alpha\in (0,1)$ of $\eqfk{k}$, we first introduce some notation. 
We set $L_\alpha := 1- \log \alpha \geq 1$ and define for $k \in \pRz$
\begin{align}
	\delfour{k} & := \int_{-k}^k \frac{1}{\vert \Mcg{c}{\denU}(t) \vert^{4}} \iwF{4}(t) d\Lm(t), \label{eq::delfour-testing}\\
	\delinf{k} & := \Vert\mathds{1}_{[-k,k]} /  \Mcg{c}{\denU} \Vert_{\Lp[\infty]{}(\iwF{})}^4 .\label{eq::delinf-testing}
\end{align} 
Due to \cref{ass::error-well-defined} both terms are well-defined and finite. More precisely, due to  the assumption $\mathds{1}_{[-k,k]}/\Mcg{c}{\denU}\in\Lp[\infty]{}(\wF)$  we have that $\mBo{k}$ is finite. It immediately follows that for any $p\in\Nz$ it holds $\mathds{1}_{[-k,k]}/\Mcg{c}{\denU}\in\Lp[2p]{}(\iwF{2p})$ for all $k\in\pRz$ since
\begin{align*}
	\int_{-k}^k \frac{1}{\vert \Mcg{c}{\denU}(t) \vert^{2p}} \iwF{2p}(t) d\Lm(t) \leq \mBo{k}^{p/2} (2k).
\end{align*}
To show bounds for the quantiles of the test statistic  we need additional moment assumptions.
\begin{ass}\label{eq::error-moments} 
	Let $\denU\in\Lp[\infty]{+}( \basMSy{2c-1})\cap\Lp[1]{+}(\basMSy{4(c-1)})$ and $\denNX\in\Lp[1]{+}(\basMSy{4(c-1)})$.
\end{ass}
It follows immediately from \cref{eq::error-moments}  that $\denU\in\Lp[1]{+}(\basMSy{2(c-1)})$. For $x,y\in\Rz$ use the notation $x\vee y := \max\{x, y\}$ and $x \wedge y := \min\{x,y\}$. Then, we write 
\begin{align}\label{eq::constants2}
    \quad \cstC[\denU] := \Vert \denU\Vert_{\Lp[\infty]{+}( \basMSy{2c-1})} / \Vert \denU\Vert_{\Lp[1]{+}(\basMSy{2(c-1)})}\vee 1,
\end{align}
which is finite under \cref{ass::error-well-defined,eq::error-moments}. If $h\in\Lp[1]{+}(\basMSy{2(c-1)})$, then  due to \cref{lem::normineq} it holds under \cref{eq::error-moments} that $h\mc \denU\in\Lp[1]{+}(\basMSy{2(c-1)})$ and we denote 
 \begin{align*}
	\cstV[h|\denU] := \Vert h \mc \denU\Vert_{\Lp[1]{+}({\basMSy{2(c-1)}})} \vee 1.
 \end{align*}
Analogously, if in addition $h\in\Lp[1]{+}(\basMSy{4(c-1)})$, then $h\mc \denU\in\Lp[1]{+}(\basMSy{4(c-1)})$ due to \cref{lem::normineq}  and we write
 \begin{align}\label{eq::const-moments}
	\icstVz[h|\denU]{} := \Vert h \mc \denU\Vert_{\Lp[1]{+}({\basMSy{4(c-1)}})} \vee 1.
 \end{align}
 Note that under \cref{eq::error-moments}, the constants $\cstV[\denNX|\denU]$ and $\icstVz[\denNX|\denU]{}$ are finite. Finally, the critical value is given for $\alpha \in(0,1)$ as
\begin{align}\label{eq::quantiletau}
	\tau_k(\alpha) := 
	\left( 18 \cstC[\denU]  \icstVz[\denNX|\denU]{}  + 69493 \frac{\sqrt{2k}}{n} \frac{L_{\alpha/2}}{\alpha} \right)  L_{\alpha/2}^{1/2} \frac{\sqrt{\delfour{k}}}{n} + 52 \cstV[\denNX|\denU]   \cstC[\denU]  L_{\alpha/2} \frac{\sqrt{\delinf{k}}}{n}.
\end{align}
The key element to analyse the behaviour of the test statistic $\etqfk{k} $ is the following decomposition:  
\begin{align}\label{eq::decompositiontest}
	\etqfk{k}  = U_{k} + 2W_{k} + \wqfk{k}(\denX-\denNX) 
\end{align}
with the canonical U-statistic 
\begin{align}\label{eq::canonicalustat}
	U_{k} := \frac{1}{n(n-1)}\sum_{\substack{j_1\not= j_2\\ j_1,j_2 \in\llbracket n\rrbracket }} \int_{-k}^k \frac{ \cfun{Y_{j_1}}{t} \cfunn{Y_{j_2}}{t}}{\vert\Mcg{c}{\denU}(t)\vert^{2} } \iwF{2}(t) d\Lm(t)
\end{align}
and the centred linear statistic
\begin{align}\label{eq::centeredlinearterm}
	W_{k} := \frac{1}{n} \sum_{j \in\llbracket n \rrbracket}  \int_{-k}^k   \frac{\cfun{Y_j}{t} (\overline{\Mcg{c}{\denY}}(t) - \overline{\Mcg{c}{\denNY}}(t))}{\vert\Mcg{c}{\denU}(t)\vert^{2}}\iwF{2}(t) d\Lm(t).
\end{align}
For the two following propositions, we apply results for the U-statistic and linear statistic, which are more technical and can be found in \cref{app::u-stat}. We first give a bound under the null hypothesis.
\begin{pr}[Bound for the quantiles of $\etqfk{k}$ under the null hypothesis]\label{pr::boundsquantiles}
	Let \cref{ass:well-definedness-testing,ass::error-well-defined,eq::error-moments} be satisfied and let $\alpha \in (0,1)$, $n\geq 2$ and $k\in\Nz$. 	Consider the estimator $\etqfk{k}$ and the threshold $\tau_k(\alpha)$ defined in \cref{eq::estimatorq2kb} and \cref{eq::quantiletau}, respectively. Under the null hypothesis we have that 
	\begin{align*}
		\ipM[\denU, \denNX]{n}(\etqfk{k} \geq \tau_k(\alpha)) \leq \alpha.
	\end{align*}
\end{pr}

\begin{proof}[Proof of \cref{pr::boundsquantiles}]
		If $\denX  = \denNX$ and, hence $\denY = \denNY$ and $\Mcg{c}{\denY} = \Mcg{c}{\denNY}$, the decomposition \cref{eq::decompositiontest} simplifies to $\etqfk{k} = U_{k}$. Due to \cref{eq::error-moments}, we have that $\denNX,\denU\in\Lp[1]{+}(\basMSy{4(c-1)})$ and the assumptions of \cref{lem::quantil-ustat} given in \cref{app::u-stat} are satisfied for $p=2$. Moreover, with $\gamma = \alpha$ and using the notations in \cref{eq::delfour-testing} and \cref{eq::delinf-testing}, we have that $\tau^{U_k}(\gamma) = \tau_k(\alpha)$. Thus, we immediately obtain the result from \cref{lem::quantil-ustat}. 
\end{proof}
Let us continue with the bound of the quantiles of the test statistic under the alternative.
\begin{pr}[Bound for the quantiles of $\etqfk{k}$ under the alternative]\label{pr::boundsquantiles-alternative}
	Let \cref{ass:well-definedness-testing,ass::error-well-defined,eq::error-moments} be satisfied and let $\beta \in (0,1)$, $n\geq 2$ and $k\in\Nz$. 	Consider the estimator $\etqfk{k}$ defined in \cref{eq::estimatorq2kb}. Under the alternative, if $\denX\in\regD\cap\Lp[1]{+}(\basMSy{4(c-1)})$ satisfies for an arbitrary critical value $\tau\in\pRz$ the separation condition 
	\begin{align}\label{eq::separationcond}
		\wqfk{k}(\denX-\denNX)  \geq 2\tau + 2\left(18\cstC[\denU] \icstVz[\denX\vert\denU]{}   
			  + 138986 \frac{L_{\beta/4}}{\beta} \frac{\sqrt{2k}}{n}\right) L_{\beta/4}^{1/2} \frac{\sqrt{\vBo{k}}}{n} + 112 \cstC[\denU] \cstV[\denX\vert\denU]\frac{L_{\beta/4}}{\beta} \frac{\sqrt{\mBo{k}}}{n},
	\end{align}	
	then, we have that
	\begin{align*}
		\ipM[\denU,\denX]{n}(\etqfk{k} < \tau) \leq \beta.
	\end{align*}
\end{pr}

\begin{proof}[Proof of \cref{pr::boundsquantiles-alternative}]
	Keeping the decomposition \cref{eq::decompositiontest} in mind, we control the deviations of the U-statistic $U_{k}$ and the linear statistic $W_{k}$ separately, applying \cref{lem::quantil-ustat} and \cref{lem::linear-stat-quantile}, respectively. The assumptions of \cref{lem::quantil-ustat} are satisfied with $p=2$. Therefore, making use of $\gamma =\beta/2$ and the notations of \cref{eq::delfour-testing} and \cref{eq::delinf-testing} we get for the quantile $\tau^{U_k}(\beta/2)$ that
	 \begin{align*}
		\tau^{U_k}(\beta/2) &= \left( 18 \icstVz[\denX|\denU]{} \cstC[\denU] + 138986 \frac{L_{\beta/4}}{\beta} \frac{\sqrt{2k}}{n} \right) L_{\beta/4}^{1/2}\frac{\sqrt{\vBo{k}}}{n}+ 52 \cstC[\denU] \cstV[\denX|\denU] L_{\beta/4} \frac{\sqrt{\mBo{k}}}{n}.
	\end{align*}
	The event 
	 $\Omega_U := \{U_{k}  \leq - \tau^{U_k}(\beta/2)\}$ satisfies
 	\begin{align}
 		\ipM[\denU,\denX]{n} (\Omega_U) &= \ipM[\denU,\denX]{n} (-U_{k}  \geq \tau^{U_k}(\beta/2))  \leq \frac{\beta}{2}\label{eq::prob-uk}
 	\end{align}
	due to \cref{lem::quantil-ustat} and the usual symmetry argument. The assumptions of \cref{lem::linear-stat-quantile} are satisfied and with $\gamma=\beta/2$ a quantile-bound of the linear statistic $W_k$ is given by
	\begin{align*}
		\tau^{W_k}(\beta/2) = 2 \frac{\sqrt{\delinf{k}}}{n}\frac{\cstC[\denU] \cstV[\denX|\denU]}{\beta}+ \frac{1}{4}q_k^2(\denX-\denNX).
	\end{align*}
	Define the event $\Omega_W := \{W_{k} <  -\tau^{W_k}(\beta/2) \}$. Due to \cref{lem::linear-stat-quantile} we have 
	\begin{align}
		\ipM[\denX,\denU]{n} (\Omega_W) =\ipM[\denX,\denU]{n} (W_{k} <  -\tau^{W_k}(\beta/2) ) \leq \frac{\beta}{2}.\label{eq::prob-wk}
	\end{align}
	Further, using $L_{\beta/4}, \frac{1}{\beta}\geq 1$ it holds that 
    \begin{align}
        &\tau^{U_k}(\beta/2)  +2\tau^{W_k}(\beta/2)\nonumber \\
		& = \left( 18 \icstVz[\denX|\denU]{} \cstC[\denU] + 138986 \frac{L_{\beta/4}}{\beta} \frac{\sqrt{2k}}{n} \right) L_{\beta/4}^{1/2}\frac{\sqrt{\vBo{k}}}{n}+ 52 \cstC[\denU] \cstV[\denX|\denU] L_{\beta/4} \frac{\sqrt{\mBo{k}}}{n} + 4 \frac{\sqrt{\delinf{k}}}{n}\frac{\cstC[\denU] \cstV[\denX|\denU]}{\beta}\nonumber\\
        & \quad  +\frac{1}{2}q_k^2(\denX-\denNX)\nonumber\\
		&\leq  \left( 18 \icstVz[\denX|\denU]{} \cstC[\denU] + 138986 \frac{L_{\beta/4}}{\beta} \frac{\sqrt{2k}}{n} \right) L_{\beta/4}^{1/2}\frac{\sqrt{\vBo{k}}}{n} + 56\cstC[\denU] \cstV[\denX|\denU] \frac{L_{\beta/4}}{\beta}\frac{\sqrt{\delinf{k}}}{n}+\frac{1}{2}q_k^2(\denX-\denNX).\label{eq::sum-taus}
    \end{align}
	Reformulating the separation condition given in \cref{eq::separationcond}, yields that
    \begin{align*}
       \wqfk{k}(\denX-\denNX)&\geq \tau + \left(18\cstC[\denU] \icstVz[\denX\vert\denU]{}   
			  + 138986 \frac{L_{\beta/4}}{\beta} \frac{\sqrt{2k}}{n}\right) L_{\beta/4}^{1/2} \frac{\sqrt{\vBo{k}}}{n} + 56 \cstC[\denU] \cstV[\denX\vert\denU]\frac{L_{\beta/4}}{\beta} \frac{\sqrt{\mBo{k}}}{n} \\
			  &\qquad + \frac{1}{2}  \wqfk{k}(\denX-\denNX).
    \end{align*}
	Plugging \cref{eq::sum-taus} into the last inequality, it follows that 
    \begin{align}
        \wqfk{k}(\denX-\denNX) \geq \tau + \tau^{U_k}(\beta/2)  +2\tau^{W_k}(\beta/2).\label{eq::ineq-taus}
    \end{align}
    Thus, combining \cref{eq::prob-uk,eq::prob-wk,eq::ineq-taus} the  decomposition given in \cref{eq::decompositiontest} implies
 	\begin{align*}
 	 &\ipM[\denX,\denU]{n} \left(\etqfk{k} < \tau \right)= \ipM[\denX,\denU]{n}\left( \left\{\etqfk{k} < \tau  \right\} \cap \Omega_U \right) + \ipM[\denX,\denU]{n} \left( \left\{\etqfk{k} < \tau  \right\} \cap \Omega_U^c \right)\\
 	&\leq\ipM[\denX,\denU]{n}(\Omega_U) +\ipM[\denX,\denU]{n}\left( 2W_{k} < \tau + \tau^{U_k}(\beta/2) - \etqfk{k}\right)\\
 	&\leq \frac{\beta}2 + \ipM[\denX,\denU]{n}(2W_{k} < -2\tau^{W_k}(\beta/2) )\\ 
 	&\leq \beta
	 \end{align*}
	which shows the claim and  completes the proof.
\end{proof}
\section{Upper bound for the radius of testing}\label{sec::uperbound-radius-testing}
For $k\in\Nz$ and $\alpha\in(0,1)$, we use the test statistic $\eqfk{k}$ defined in \cref{eq::estimatorq2kb} and the critical value $\tau_k(\alpha)$ given in \cref{eq::quantiletau} to define the test
\begin{align}\label{eq::test-stat2}
\Delta_{k,\alpha} := \mathds{1}_{\{\eqfk{k}  \geq \tau_k(\alpha) \}}.
\end{align}
From \cref{pr::boundsquantiles} follows immediately that the test $\Delta_{k,\alpha}$
is a level-$\alpha$-test for all $k\in\Nz$. To analyse its power over the alternative, we consider the regularity class $\regC$ defined in \cref{eq::reg-class} with regularity parameter $\operatorname{s}$ satisfying \cref{ass::classes-testing} and  $R\in\pRz$.
\cref{pr::boundsquantiles-alternative} allows to characterize elements in $\regC$ for which $\Delta_{k,\alpha} $ is powerful. Exploiting these results, we derive in this section an upper bound for the radius of testing of $\Delta_{k,\alpha} $ in terms of $\delfour{k}$ and $\delinf{k}$ defined in \cref{eq::delinf-testing,eq::delfour-testing}, respectively, and the regularity parameter $\operatorname{s}$. That is, define
\begin{align*}
	\rho^2_{k,\operatorname{s}} := \rho^2_{k,\operatorname{s}}(n) :=\frac{\iwF[]{2}(k)}{\operatorname{s}^2(k)}  \vee \frac{1}{n} \left(\sqrt{\delfour{k}} \vee \sqrt{\delinf{k}}\right).
\end{align*}
For the result on the upper bound of the radius of testing, we need that the additional moment assumption for the hypothesis (\cref{eq::error-moments}) also holds uniformly for the elements of the alternative. More precisely, we introduce for $\operatorname{v}\in\Rz$, $\text{v}{\geq 1}$ the set
\begin{align*}
	\regDv := \{ h\in\regD \cap \Lp[1]{+}(\basMSy{4(c-1)}):\icstVz[h|\denU]{} \leq \operatorname{v} \}.
\end{align*}
The constant $\icstVz[h|\denU]{}$ is defined in \cref{eq::const-moments} and finite for $h\in\Lp[1]{+}(\basMSy{4(c-1)})$. The following result gives a bound on the maximal risk of the proposed test.
\begin{pr}[Upper bound for the radius of testing of $\Delta_{k,\gamma/2}$]\label{pr::upperbound-radius-testing}
	Under \cref{ass::classes-testing,ass:well-definedness-testing,ass::error-well-defined,eq::error-moments}, for $\gamma\in(0,1)$ and $\operatorname{d}\in\pRz$ define 
	\begin{align}
		A_\gamma^2 := R^2 + 140 \frac{L_{\gamma/8}}{\gamma} \cstC[\denU]\icstVz[\denNX|\denU]{} + 260 \frac{L_{\gamma/8}}{\gamma} \cstC[\denU]\operatorname{v} + 833934 \frac{L^{3/2}_{\gamma/8}}{\gamma} \operatorname{d}\label{eq::A-def}
	\end{align}
	and
	\begin{align*}
		\eta_k^2 := \eta_k^2(n) := 1 \vee \frac{\sqrt{2k}}{\operatorname{d}n}.
	\end{align*}
	For all $A\geq A_\gamma$ and $n,k\in\Nz$ with $n\geq2$  we have
	\begin{align*}
		\mathcal{R} \left( \Delta_{k, \gamma/2} \vert 	\regDv, \regC ,A^2\rho_{k,\operatorname{s}}^2\cdot \eta_k^2 \right) \leq \gamma.
	\end{align*}
\end{pr}
\begin{proof}[Proof of \cref{pr::upperbound-radius-testing}] 
	The result follows from \cref{pr::boundsquantiles} and \cref{pr::boundsquantiles-alternative}  with $\alpha = \beta = \gamma/2$, $\tau = \tau_k(\alpha)$, and the definition of the maximal risk
	\begin{align*}
		\mathcal{R} \left( \Delta_{k, \gamma/2} \vert 	\regDv, \regC ,A^2\rho_{k,\operatorname{s}}^2\cdot \eta_k^2 \right)  &= \ipM[\denU,\denNX]{n}\left(  \Delta_{k, \gamma/2}= 1 \right) \\
		&\quad+ \sup_{\denX \in\regDv\cap (\denNX + \regC \cap \regF_{\geq A^2\rho_{k,\operatorname{s}}^2\cdot \eta_k^2})} \ipM[\denU,\denX]{n}\left(  \Delta_{k, \gamma/2} = 0\right)\\
		&\leq \frac{\gamma}{2} + \frac{\gamma}{2} = \gamma.
	\end{align*}
	Since \cref{ass:well-definedness-testing,ass::error-well-defined,eq::error-moments} are satisfied, we can apply \cref{pr::boundsquantiles}. To conclude the proof, we need  to check the assumptions of \cref{pr::boundsquantiles-alternative}, that is, to verify \cref{eq::separationcond} for elements in the alternative $\denX \in\regDv\cap (\denNX + \regC \cap \regF_{\geq A^2\rho_{k,\operatorname{s}}^2\cdot \eta_k^2})$.
	More precisely, under \cref{ass::classes-testing,ass:well-definedness-testing,ass::error-well-defined,eq::error-moments} it remains to show
	\begin{align}
		\wqfk{k}(\denX-\denNX)  &\geq 2\tau_k(\gamma/2) + 2\left(18\cstC[\denU] \icstVz[\denX\vert\denU]{}   
			  + 2\cdot 138986 \frac{L_{\gamma/8}}{\gamma} \frac{\sqrt{2k}}{n}\right) L_{\gamma/8}^{1/2} \frac{\sqrt{\vBo{k}}}{n} \nonumber\\
			  &\quad + 2\cdot 112 \cstC[\denU] \cstV[\denX\vert\denU]\frac{L_{\gamma/8}}{\gamma} \frac{\sqrt{\mBo{k}}}{n}.\label{eq::bound-on-A}
	\end{align}
	Indeed, with \cref{ass::classes-testing} and  $\denX-\denNX\in\regC $, we bound the bias under the alternative. More precisely, it holds that
	\begin{align}
		\wqf(\denX-\denNX) - \wqfk{k}(\denX-\denNX) &= \Vert \mathds{1}_{[-k,k]^c} (\Mcg{c}{\denX}- \Mcg{c}{\denNX}) \Vert_{\Lp[2]{}(\iwF[]{2})}^2\nonumber\\
		&= \Vert \mathds{1}_{[-k,k]^c}(\Mcg{c}{\denX}- \Mcg{c}{\denNX}) {\operatorname{s}} / {\operatorname{s}} \Vert_{\Lp[2]{}(\iwF[]{2})}^2 \nonumber\\
		&\leq \Vert \mathds{1}_{[-k,k]^c}/ {\operatorname{s}}\Vert_{\Lp[\infty]{}(\wF)}^2 \Vert\operatorname{s} (\Mcg{c}{\denX}- \Mcg{c}{\denNX}) \Vert_{\Lp[2]{}(\iwF[]{2})}^2\nonumber\\
		&\leq \frac{\iwF[]{2}(k)}{\operatorname{s}^2(k)}  R^2.\label{eq::bias-calc1}
	\end{align}
	Consequently, we have for $\denX -\denNX\in\regC \cap \regF_{\geq A^2\rho_{k,\operatorname{s}}^2\cdot \eta_k^2}$ that
	\begin{align}
				\wqfk{k}(\denX-\denNX) &\geq \wqf(\denX-\denNX) - \frac{\iwF[]{2}(k)}{\operatorname{s}^2(k)}  R^2\geq A^2\rho_{k,\operatorname{s}}^2\cdot  \eta_k^2- \frac{\iwF[]{2}(k)}{\operatorname{s}^2(k)} R^2.\label{eq::bias-calculation}
	\end{align}
	Let us upper bound the critical value $\tau_{k}(\gamma/2)$. Since $L_{\gamma} =  1- \log \gamma \geq 1$ for all $\gamma\in(0,1)$ and $L_{\gamma}$ is increasing in $\gamma$, we also have that $L_{\gamma/4},L_{\gamma/4}^{1/2}\leq L_{\gamma/8}$. Thus, it follows
	\begin{align*}
		 L_{\gamma/4} L^{1/2}_{\gamma/4}\leq {L_{\gamma/8}^{3/2}}.
	\end{align*}
	Additionally, because $\cstC[\denU]  \cstV[\denX|\denU] \leq \icstC[\denU]{2}  \icstVz[\denX|\denU]{}$
	we obtain for the critical value
	\begin{align}
		\tau_{k}(\gamma/2) &= \left( 18 \cstC[\denU]  \icstVz[\denNX|\denU]{}  + 138986 \frac{\sqrt{2k}}{n} \frac{L_{\gamma/4}}{\gamma} \right)  L_{\gamma/4}^{1/2} \frac{\sqrt{\delfour{k}}}{n} + 52 \cstV[\denNX|\denU]   \cstC[\denU]  L_{\gamma/4} \frac{\sqrt{\delinf{k}}}{n}\nonumber\\
		&\leq \left( 18  \frac{L_{\gamma/8}}{\gamma}\cstC[\denU]   \icstVz[\denNX|\denU]{}+ 138986\frac{\sqrt{2k}}{n} \frac{L_{\gamma/8}^{3/2}}{\gamma}\right)  \frac{\sqrt{\delfour{k}}}{n}+ 52 \cstV[\denNX|\denU]   \cstC[\denU]  \frac{L_{\gamma/8}}{\gamma} \frac{\sqrt{\delinf{k}}}{n}\nonumber\\
		&\leq \left(70  \icstVz[\denNX|\denU]{}   \cstC[\denU]  \frac{L_{\gamma/8}}{\gamma} + 138986\frac{\sqrt{2k}}{n} \frac{L_{\gamma/8}^{3/2}}{\gamma} \right) \left(\frac{\sqrt{\delfour{k}}}{n}\vee\frac{\sqrt{\delinf{k}}}{n} \right).\label{eq::bound-qt}
	\end{align}
	Using $A\geq A_\gamma$ for $A_\gamma$ defined in \cref{eq::A-def} and  $\operatorname{d} \eta_k^2\geq \frac{\sqrt{2k}}{n}$ , we get for  \cref{eq::bias-calculation} that 
	\begin{align*}
		&\wqfk{k}(\denX-\denNX) 
		\geq \left( R^2 + 140 \frac{L_{\gamma/8}}{\gamma} \cstC[\denU]\icstVz[\denNX|\denU]{} + 260 \frac{L_{\gamma/8}}{\gamma} \cstC[\denU]\operatorname{v} + 833934 \frac{L^{3/2}_{\gamma/8}}{\gamma}\eta_k^2 \operatorname{d}\right) \\
		&\quad\quad \quad \quad\quad \quad \qquad \cdot\left(\frac{\iwF[]{2}(k)}{\operatorname{s}^2(k)} \vee \frac{1}{n} \left(\sqrt{\delfour{k}} \vee \sqrt{\delinf{k}}\right)\right) -  \frac{\iwF[]{2}(k)}{\operatorname{s}^2(k)}  R^2\\
		 &\geq \left((2\cdot 70) \frac{L_{\gamma/8}}{\gamma} \cstC[\denU]\icstVz[\denNX|\denU]{}+ (2\cdot 18 + 2\cdot 112) \frac{L_{\gamma/8}}{\gamma} \cstC[\denU]\icstVz[\denX|\denU]{}+ (2+4) 138986\frac{L^{3/2}_{\gamma/8}}{\gamma} \frac{\sqrt{2k}}{n}\right)  \\
		&\quad\quad \quad \cdot \frac{1}{n} \left(\sqrt{\delfour{k}} \vee \sqrt{\delinf{k}}\right).\\
		 \intertext{Combining the calculation with the bound derived in \cref{eq::bound-qt} for the critical value $\tau_{k}(\gamma/2)$, we finally obtain}
		&\wqfk{k}(\denX-\denNX)  \geq 2\tau_k(\gamma/2) + 2\left(18\cstC[\denU] \icstVz[\denX\vert\denU]{} + 2\cdot 138986 \frac{L_{\gamma/8}}{\gamma} \frac{\sqrt{2k}}{n}\right) L_{\gamma/8}^{1/2} \frac{\sqrt{\vBo{k}}}{n} \nonumber\\
		&\quad\quad \quad + 2\cdot 112 \cstC[\denU] \cstV[\denX\vert\denU]\frac{L_{\gamma/8}}{\gamma} \frac{\sqrt{\mBo{k}}}{n},
	\end{align*}
	which shows the separation condition given in \cref{eq::bound-on-A} and completes the proof.
\end{proof}

Let us introduce a dimension parameter that realizes an optimal bias-variance trade-off and the corresponding radius:
\begin{align}\label{eq::parameter-opt}
	k^*_{\operatorname{s}} := k^*_{\operatorname{s}}(n) := \arg\min_{k\in\Nz} \rho^2_{k,\operatorname{s}}(n) \quad\text{ and }\quad \rho^2_{*,\operatorname{s}}:= \rho^2_{*,\operatorname{s}}(n) = \min_{k\in\Nz} \rho^2_{k,\operatorname{s}}(n).
\end{align}

The next result follows immediately from 
\cref{pr::upperbound-radius-testing} and, hence, we omit the proof.
\begin{co}\label{co::optimal-radius-testing}
	Under the assumptions of \cref{pr::upperbound-radius-testing} let $\gamma\in(0,1)$ and $A_\gamma$ as in \cref{eq::A-def}. For all $A\geq A_\gamma$ and $n\geq2$ we have
	\begin{align*}
		\mathcal{R} \left( \Delta_{k^*_{\operatorname{s}}, \gamma/2} \vert \regDv,\regC ,A^2\rho_{*,\operatorname{s}}^2\cdot \eta_{k^*_{\operatorname{s}}}^2 \right) \leq \gamma.
	\end{align*}
\end{co}

We note that \cref{co::optimal-radius-testing} derives $\rho_{*,\operatorname{s}}^2\cdot \eta_{k^*_{\operatorname{s}}}^2$ as an upper bound of the separation radius for any finite sample size $n\geq 2$. Typically, $\eta_{k^*_{\operatorname{s}}}^2$ is a deterioration which is asymptotically negligible. To be more precise, we study the asymptotic behaviour of this upper bound as the sample tends to infinity by determining the order of
	\begin{align*}
		\rho^2_{*,\operatorname{s}}(n)  = \min_{k\in\Nz} \left(\frac{\iwF[]{2}(k)}{\operatorname{s}^2(k)}\vee \frac{1}{n} \left(\sqrt{\delfour{k}} \vee \sqrt{\delinf{k}}\right)\right).
	\end{align*}
Following  \cite{Comte2025}, we consider different choices of regularity function $\operatorname{s}$, the density function $\iwF[]{2}$  and the order of the Mellin transform of the error density, which determine the order of $\delfour{k}$ and  $\delinf{k}$. 

For two real-valued functions $h_1,h_2\colon\Rz\rightarrow\Rz $ we write $h_1(t)\sim h_2(t)$ if there exist constants $C,\tilde{C}$ such that for all $t\in\Rz$ we have $h_1(t)\leq C h_2(t)$ and $h_2(t)\leq \tilde{C} h_1(t)$. We then call $h_1$ and $h_2$ of the same order.
The class $\regC$ defined in \cref{eq::reg-class} covers the usual assumptions on the regularity of the unknown density $\denX$, i.e., ordinary and super smooth densities. More precisely, in the \textit{ordinary smooth case} the regularity function $\operatorname{s}$ is of polynomial order, i.e., $\operatorname{s}(t) \sim  (1+t^2)^{s/2}$ for some $s\in\pRz$. In this case, $\regC$ corresponds to the Mellin-Sobolev space, see Definition 2.3.9 in \cite{BrennerMiguelDiss}. 
In the \textit{super smooth} case, $\operatorname{s}$ is assumed to be exponentially increasing, i.e., $\operatorname{s}(t) \sim \exp(\vert t\vert^s)$ for some $s\in\pRz$.

Recall that the quadratic functional $q^2(h)$ is finite for all elements $h$ in the nonparametric class of functions $\regC$ if $\wF / \operatorname{s}$  is bounded (\cref{ass::classes-testing} (ii)). Convergence of $\wF / \operatorname{s}$ against zero as $\vert t\vert \rightarrow \infty$ gives the rate for the bias term. 
Furthermore, we restrict ourselves to the case that $\iwF{2}(t) \sim (1+t^2)^{a}$ for $a\in\Rz$. This includes the examples given in \cref{il::example-parameters} with the following choices of $a$.
\begin{itemize}
		\item[(i)] In the case $a=0$ we have $\wF(t) = 1$ for all $t\in\Rz$. Recall that in this case $\wqf(f)$ equals a (weighted) $\Lp[2]{+}$-norm of the density $\denX$ itself.

		\item[(ii)] The case  $a=-1$ covers quadratic functional estimation of the survival function $S_f$ of $X$, more precisely, we have $\iwF{2}(t)= \frac{1}{(c-1)^2+4\pi t^2}$ for all $t\in\Rz$.

		\item[(iii)] The case $a={\beta\in\Nz}$ covers quadratic functionals of derivatives $D^\beta[\denX]$ of the density $\denX$. 

\end{itemize}
Note that, only in case of ordinary smoothness of the unknown density $\denX$, \cref{ass::classes-testing} (ii) imposes the additional condition $s>a$ on the parameters.

We consider in the following two cases for the behaviour of the Mellin transform of the error density $\denU$. We either assume for some decay parameter $\parM\in\pRz$  its \textit{(ordinary) smoothness}, i.e.,  $\vert \Mcg{c}{\denU}(t) \vert \sim (1+t^2)^{-\parM/2}$;
or its \textit{super smoothness}, i.e., $\vert \Mcg{c}{\denU}(t)  \vert \sim \exp( - \vert t\vert^{\sigma})$. Under these configurations standard calculations allow to derive the corresponding rates  given in \cref{tab:risks-general-testing}.

For the case that both densities are ordinary smooth (first line), we assume that $\sigma +a > -\frac14$. Let us remark that in the case of $\parM +a = -\frac14$ the term $\sqrt{\vBo{k}}\vee\sqrt{\mBo{k}} = \sqrt{\vBo{k}}$ is of order $\sqrt{\log k}$ and the corresponding rate of testing is of order $\sqrt{\log n}/{n}$. If the density function $\wF$ is even of smaller order, i.e., $\parM +a < -\frac14$, it holds that $\sqrt{\vBo{k}}\vee\sqrt{\mBo{k}}\sim 1$ and we achieve a parametric rate $n^{-1}$.
If the unknown density function is ordinary smooth and the error density is super smooth (second line in \cref{tab:risks-general-testing}), then, the optimal parameter $k^*_{\operatorname{s}}$ does not depend on the regularity parameter $s$. Hence, in this case the test $ \Delta_{k^*_{\operatorname{s}}}$ is automatically adaptive. 
	We shall emphasize that in all three examples the factor
	\begin{align*}
	\eta_{k^*_{\operatorname{s}}}^2 = 1\vee \frac{\sqrt{2k^*_{\operatorname{s}}}}{\operatorname{d}n}
	\end{align*}
	is of order one and, thus, the order of the radius of testing is indeed equal to the order of $\rho^2_{*,\operatorname{s}}(n) $.

\begin{table}[!ht]
	\centering
	\begin{tabular}{c c c c}
		\toprule
		$\operatorname{s}(t)$&$\vert \Mcg{c}{\denU}(t)  \vert $ & $k^*_{\operatorname{s}}$ & $\rho^2_{*,\operatorname{s}}(n) $ \\
		\midrule
		$(1+t^2)^{\frac{s}{2}}$ & $(1+t^2)^{-\frac{\parM}{2}}$ & $n^{\frac{2}{4s+4\parM +1}}$  & $n^{-\frac{4(s-a)}{4s+4\parM +1}}$\\
		$(1+t^2)^{\frac{s}{2}}$  & $\exp( - \vert t\vert^{\sigma})$&$(\log n)^{\frac{1}{\sigma}}$ & $(\log n)^{-\frac{2(s-a)}{\sigma}}$ \\
		$\exp(\vert t\vert^s)$& $(1+t^2)^{-\frac{\parM}{2}}$ &$(\log n)^{\frac{1}{s}}$  &$\frac{1}{n} (\log n)^{\frac{2(\parM+a)+1/2}{s}}$\\
		\bottomrule
	\end{tabular}
	\caption{Order of the radius of testing for $\iwF{2}(t) \sim (1+t^2)^{a}$.}
	\label{tab:risks-general-testing}
\end{table}

\section{Testing radius of a max-test}\label{sec::radius-max-test}
The performance of the test $ \Delta_{k, \gamma/2}$ in \cref{pr::upperbound-radius-testing} relies on the choice of the dimension parameter $k$ and its optimal selection in turn depends generally on the alternative class $\regC$, see \cref{tab:risks-general-testing}. A testing procedure is called adaptive, i.e., assumption-free, if it performs optimally for a wide range of alternatives. In this section, we therefore propose a data-driven testing procedure by aggregating the tests derived in the last section over various dimension parameters $k$ and discuss assumptions under which it performs (nearly) adaptive. In \cref{sec::data-driv-bon}, we first introduce the testing procedure and then derive in \cref{sec::max-rad} an upper bound for its radius of testing.

\subsection{\textit{Data-driven procedure via Bonferroni aggregation}}\label{sec::data-driv-bon}

We first recapitulate the main ideas of the Bonferroni aggregation. Let $\mathcal{K}\subseteq\Nz$ be a finite collection of dimension parameters. Recall that by \cref{pr::boundsquantiles} for $\alpha \in (0,1)$ each test in the following family 
\begin{align*}
	(\Delta_{k,\alpha/\vert \mathcal{K}\vert})_{k\in\mathcal{K}} := (\mathds{1}_{\{\eqfk{k}  \geq \tau_k(\alpha/\vert \mathcal{K}\vert) \}})_{k\in\mathcal{K}}
\end{align*}
has level $ \alpha/\vert \mathcal{K}\vert\in(0,1)$. We consider the max-test 
\begin{align}
	\Delta_{\mathcal{K},\alpha} := \mathds{1}_{\{\zeta_{\mathcal{K},\alpha} >0\}} \quad\text{ with }\quad \zeta_{\mathcal{K},\alpha} = \max_{k\in\mathcal{K}} \Delta_{k,\alpha/\vert \mathcal{K}\vert}\label{eq::max-test}
\end{align}
i.e. the test rejects the null hypothesis as soon as one of the tests in the collection does. Under the null hypothesis, we bound the type I error probability of the max-test by the sum of the error probabilities of the individual tests
\begin{align}\label{eq::type-one-error}
	\ipM[\denU,\denNX]{n}\left(  \Delta_{\mathcal{K},\alpha}  = 1\right) =\ipM[\denU,\denNX]{n}\left( \zeta_{\mathcal{K},\alpha} >0\right) \leq \sum_{k\in\mathcal{K}} \ipM[\denU,\denNX]{n} (\Delta_{k,\alpha/\vert \mathcal{K}\vert} = 1)\leq \sum_{k\in\mathcal{K}} \frac{\alpha}{\vert \mathcal{K}\vert} = \alpha.
\end{align}
Hence, $\Delta_{\mathcal{K},\alpha}$ is a level-$\alpha$-test. Under the alternative the type II error probability is bounded by the minimum of the error probabilities of the individual tests, i.e.
\begin{align}\label{eq::type-two-error}
	\ipM[\denU,\denX]{n}\left(  \Delta_{\mathcal{K},\alpha}  = 0\right) =\ipM[\denU,\denX]{n}\left( \zeta_{\mathcal{K},\alpha} \leq 0\right) \leq \min_{k\in\mathcal{K}}\ipM[\denU,\denX]{n} (\Delta_{k,\alpha/\vert \mathcal{K}\vert} = 0).
\end{align}
Therefore, $\Delta_{\mathcal{K},\alpha}$ has the maximal power achievable by a test in the collection. The bounds \cref{eq::type-one-error} and \cref{eq::type-two-error} have opposing effects on the choice of the collection $\mathcal{K}$. On the one hand, it should be as small as possible to keep the type I error probability small. On the other hand, it must be large enough to approximate an optimal dimension parameter for a wide range of alternative classes that we want to adapt to. In this work we consider a classical Bonferroni choice of an error level $\alpha/\vert \mathcal{K}\vert$. For other aggregation choices, e.g. Monte Carlo quantile and Monte Carlo threshold method, we refer to \cite{Baraud_2003} and \cite{Fromont_2006}. Although the Bonferroni choice is a more conservative method, its optimality is often shown, which is then also shared with the other methods.

Applying the Bonferroni choice to the level $\alpha/\vert \mathcal{K}\vert\in(0,1)$, the critical value in \cref{eq::quantiletau} writes
\begin{align*}
\tau_k(\alpha/\vert \mathcal{K}\vert) := 
	\left( 18 \cstC[\denU]  \icstVz[\denNX|\denU]{}  + 69493 \frac{\sqrt{2k}}{n} \frac{L_{\alpha/2\vert \mathcal{K}\vert}\vert \mathcal{K}\vert}{\alpha} \right)  L_{\alpha/2\vert \mathcal{K}\vert}^{1/2} \frac{\sqrt{\delfour{k}}}{n} + 52 \cstV[\denNX|\denU]   \cstC[\denU]  L_{\alpha/2\vert \mathcal{K}\vert} \frac{\sqrt{\delinf{k}}}{n}.
\end{align*}
We observe that the critical value and hence the separation condition \cref{eq::separationcond} in \cref{pr::boundsquantiles-alternative} increases at least by a factor $\vert \mathcal{K}\vert$ leading possibly to a deterioration of the radius of testing by this factor. This is known to be suboptimal. To find a sharper bound on the critical value, we impose in this section a stronger moment condition under the null hypothesis. That is, instead of \cref{eq::error-moments}, we assume in the results of this section the following.
\begin{ass}\label{ass::error-p-moments}
	For $p\geq 2$  let $\denU\in\Lp[\infty]{+}( \basMSy{2c-1})\cap\Lp[1]{+}(\basMSy{2p(c-1)})$ and $\denNX\in\Lp[1]{+}(\basMSy{2p(c-1)})$.
\end{ass}
We shall emphasize that the parameter $p$ given in \cref{ass::error-p-moments} is known under the null hypothesis since the error density is known. However, 
under \cref{ass::error-p-moments} , it follows that $\denNX\mc\denU\in\Lp[1]{+}(\basMSy{2p(c-1)})$ by \cref{lem::normineq} and we write shortly, 
\begin{align*}
\cstV[\denNX|\denU](p) := 	 \Vert\denNX\mc\denU\Vert_{\Lp[1]{+}(\basMSy{2p(c-1)})} \vee 1.
\end{align*}
Finally, for $\alpha\in (0,1)$ we introduce an updated critical value 
\begin{align}
	\tau_{k\vert\alpha} &:= 
	\left( 18 \cstC[\denU]  \cstV[\denNX|\denU](p)  + 69493 \frac{\sqrt{2k}\vert \mathcal{K}\vert^{1/(p-1)}}{n} \frac{L_{\alpha/2\vert\mathcal{K}\vert}^2}{\alpha} \right)  L_{\alpha/2\vert\mathcal{K}\vert}^{1/2} \frac{\sqrt{\delfour{k}}}{n}\nonumber\\
	&\qquad + 52 \cstV[\denNX|\denU]   \cstC[\denU]  L_{\alpha/2\vert\mathcal{K}\vert} \frac{\sqrt{\delinf{k}}}{n}.\label{eq::quantiletau-new}
\end{align}
Below  we generalize \cref{pr::boundsquantiles} to the updated critical value and, therewith, we show that each test in the following family 
\begin{align*}
	(\Delta_{k\vert \alpha})_{k\in\mathcal{K}} := (\mathds{1}_{\{\eqfk{k}  \geq \tau_{k\vert\alpha} \}})_{k\in\mathcal{K}}
\end{align*}
has level $ \alpha/\vert \mathcal{K}\vert\in(0,1)$. Afterwards, we consider the updated max-test 
\begin{align}
	\Delta_{\mathcal{K}\vert\alpha} := \mathds{1}_{\{\zeta_{\mathcal{K}\vert \alpha} >0\}} \quad\text{ with }\quad \zeta_{\mathcal{K}\vert \alpha} = \max_{k\in\mathcal{K}} \Delta_{k\vert \alpha}\label{eq::max-test-updated}
\end{align}
and derive an upper bound of its radius of testing.

\begin{pr}[Bound for the quantiles of $\etqfk{k}$ under the null hypothesis]\label{pr::boundsquantiles-updated}
	Let \cref{ass:well-definedness-testing,ass::error-well-defined,ass::error-p-moments} be satisfied and let $\alpha \in (0,1)$, $n\geq 2$, $k\in\Nz$ and $\mathcal{K}\subseteq\Nz$ be finite. 	Consider the estimator $\etqfk{k}$ and the threshold $\tau_{k\vert\alpha}$ defined in \cref{eq::estimatorq2kb} and \cref{eq::quantiletau-new}, respectively. Under the null hypothesis we have that 
	\begin{align*}
		\ipM[\denU, \denNX]{n}(\etqfk{k} \geq \tau_{k\vert\alpha}) \leq \frac{\alpha}{\vert\mathcal{K}\vert}.
	\end{align*}
\end{pr}
\begin{proof}[Proof of \cref{pr::boundsquantiles-updated}]
	The proof follows along the same lines as the proof of \cref{pr::boundsquantiles}. If $\denX  = \denNX$ and, hence $\denY = \denNY$ and $\Mcg{c}{\denY} = \Mcg{c}{\denNY}$, the decomposition \cref{eq::decompositiontest} simplifies to $\etqfk{k} = U_{k}$. Due to \cref{ass::error-p-moments}, we have that $\denNX,\denU\in\Lp[1]{+}(\basMSy{2p(c-1)})$ and the assumptions of \cref{lem::quantil-ustat} are satisfied. We intend to apply \cref{lem::quantil-ustat} with $\gamma = \alpha/\vert\mathcal{K}\vert$.  Moreover, using the notations in \cref{eq::delfour-testing} and \cref{eq::delinf-testing} and the fact that for $p\geq 2$ it holds that $ \alpha^{-1/(p-1)} \leq \alpha^{-1}$ and $ L_{\alpha/2\vert\mathcal{K}\vert}^{2 - 1/(p-1)} \leq L_{\alpha/2\vert\mathcal{K}\vert}^2$, we have that 
	\begin{align*}
		\tau^{U_k}( \alpha/\vert\mathcal{K}\vert)  &= \left( 18 \cstV[\denNX|\denU](p) \cstC[\denU] + 69493   \frac{\sqrt{2k}\vert\mathcal{K}\vert^{1/(p-1)}}{n \alpha^{1/(p-1)}} L_{\alpha/2\vert\mathcal{K}\vert}^{2 - 1/(p-1)}  \right)\frac{\sqrt{\vBo{k}}}{n} L_{\alpha/2\vert\mathcal{K}\vert}^{1/2}\nonumber\\
		&\quad + 52 \cstC[\denU] \cstV[\denNX|\denU] \frac{\sqrt{\mBo{k}}}{n} L_{\alpha/2\vert\mathcal{K}\vert}\leq \tau_{k\vert\alpha}.
	\end{align*}
	 Thus, we obtain immediately from \cref{lem::quantil-ustat} that 
	 \begin{align*}
		\ipM[\denU, \denNX]{n}(\etqfk{k} \geq \tau_{k\vert\alpha}) \leq \ipM[\denU, \denNX]{n}(\etqfk{k} \geq 	\tau^{U_k}( \alpha/\vert\mathcal{K}\vert))\leq \frac{\alpha}{\vert\mathcal{K}\vert}.
	 \end{align*} 
	 This concludes the proof.
\end{proof}
In summary, due to \cref{pr::boundsquantiles-updated} and \cref{eq::type-one-error}, for any finite collection $\mathcal{K}\subset\Nz$ and any $\alpha\in(0,1)$ the max-test given in \cref{eq::max-test-updated} is of level $\alpha$.
In the next section, we derive an upper bound for its radius of testing.

\subsection{\textit{Testing radius of the max-test}}\label{sec::max-rad}

 Denote by $\mathcal{S}$ the set of regularity functions satisfying \cref{ass::classes-testing}. The set $\mathcal{S}$ characterizes the collection of alternatives $\{\regC: \operatorname{s}\in\mathcal{S}\}$ for which we analyse the power of the testing procedure simultaneously.
 The max-test only aggregates over a finite set $\mathcal{K}\subseteq\Nz$. We define the minimal achievable radius of testing for each $\operatorname{s}\in\mathcal{S}$ over the set $\mathcal{K}$  for each $x\in \pRz$ as
\begin{align*}
	\rho^2_{\mathcal{K}, \regs}(x) := \min_{k\in\mathcal{K}} \rho_{k,\regs}^2(x) \quad \text{with} \quad \rho_{k,\regs}^2 (x):= \frac{\iwF{2}(k)}{\operatorname{s}^2(k)} \vee \frac{1}{x} \left(\sqrt{\delfour{k}} \vee \sqrt{\delinf{k}}\right).
\end{align*}
Since $\rho^2_{*,\regs} =  \rho_{k^*_{\operatorname{s}},\regs}^2(n) $ is defined in \cref{eq::parameter-opt} as the minimum taken over $\Nz$ instead of $\mathcal{K}$, for each $n\in\Nz$ we always have $\rho^2_{\mathcal{K}, \regs}(n) \geq \rho^2_{\Nz, \regs}(n) = \rho_{*,\regs}^2$ and evidently $\rho^2_{\mathcal{K}, \regs}(n) = \rho^2_{\Nz, \regs}(n)$ whenever $k^*_{\operatorname{s}}\in\mathcal{K}$. Let us introduce the factor $\delta_\mathcal{K}:= (1 + \log \vert\mathcal{K}\vert)^{-1/2}\in(0,1)$. For each $\operatorname{s}\in\mathcal{S}$, $k\in\mathcal{K}$ and $n\in\Nz$ we set
\begin{align}
	\operatorname{r}^2_{k, \operatorname{s}} := \operatorname{r}^2_{k, \operatorname{s}} (n)&:= \frac{\iwF{2}(k)}{\operatorname{s}^2(k)} \vee \frac{1}{n} \left(\delta_\mathcal{K}^{-1}\sqrt{\delfour{k}} \vee \delta_\mathcal{K}^{-2}\sqrt{\delinf{k}}\right)\nonumber\\
	k_{\mathcal{K}, \operatorname{s}}:= k_{\mathcal{K}, \operatorname{s}} (n) &:= \arg\min_{k\in\mathcal{K}} \operatorname{r}^2_{k, \operatorname{s}} (n)\label{eq::choice-k}\\
	\operatorname{r}_{\mathcal{K}, \operatorname{s}}^2:= \operatorname{r}_{\mathcal{K}, \operatorname{s}}^2(n) &:=  \min_{k\in\mathcal{K}}\operatorname{r}^2_{k, \operatorname{s}} (n).\label{eq::max-radius}
\end{align}
%\textcolor{red}{Discussion adaptive factor? Quadrate?}

Evidently, since $\delta_{\mathcal{K}}\in (0,1)$ we have
\begin{align*}
	\rho^2_{\mathcal{K}, \regs}(\delta_{\mathcal{K}}n) \leq \operatorname{r}_{\mathcal{K}, \operatorname{s}}^2(n)\leq \rho^2_{\mathcal{K}, \regs}(\delta_{\mathcal{K}}^2n).
\end{align*}

Provided that the collection $\mathcal{K}$ is chosen such that $\rho^2_{\mathcal{K}, \regs}(\delta_{\mathcal{K}}n) \asymp \rho^2_{\Nz, \regs}(\delta_{\mathcal{K}})$ and $\rho^2_{\mathcal{K}, \regs}(\delta_{\mathcal{K}}^2n)\asymp \rho^2_{\Nz, \regs}(\delta_{\mathcal{K}}n)$ in the best case the effective sample size is $\delta_\mathcal{K}n$ and in the worst case $\delta_\mathcal{K}^2n$. Consequently, $\delta_{\mathcal{K}}$ is also called adaptive factor.

In the next proposition we show for each $\operatorname{s}\in\mathcal{S}$,  by applying \cref{pr::boundsquantiles-updated,pr::boundsquantiles-alternative} with $\alpha = \beta = \gamma/2$, that both the type I and maximal type II error probability of the updated max-test $\Delta_{\mathcal{K}\vert\gamma/2}$ are bounded by $\gamma/2$ and thus their sum is bounded by $\gamma$. 

\begin{pr}[Uniform radius of testing over regularity class $\mathcal{S}$]\label{pr::uniform-testing-radius}
Let  \cref{ass:well-definedness-testing,ass::error-well-defined,ass::error-p-moments} be satisfied. For $\gamma \in (0,1)$ and $\operatorname{d}\in\Rz, \operatorname{d}\geq 1$ define
\begin{align*}
	A_\gamma^2 :=  R^2 + 140 \frac{L_{\gamma/8}}{\gamma} \cstC[\denU]\cstV[\denNX|\denU](p) + 260 \frac{L_{\gamma/8}}{\gamma} \cstC[\denU]\operatorname{v} + 833934 \frac{L^{5/2}_{\gamma/8}}{\gamma} \operatorname{d}.
\end{align*}
 For  $n\geq 2$, $\mathcal{K}\subseteq\Nz$ finite and $\operatorname{s}\in\mathcal{S}$ with $k_{\mathcal{K}, \operatorname{s}} \in\mathcal{K}$ as defined in \cref{eq::choice-k} set
\begin{align}\label{eq::as-prop-unif}
	\eta_{\mathcal{K},\regs}^2:= \eta_{\mathcal{K},\regs}^2(n) :=  1 \vee \frac{\sqrt{2k_{\mathcal{K}, \operatorname{s}} }\vert \mathcal{K}\vert^{1/(p-1)}}{ \operatorname{d}n \delta_{\mathcal{K}}^4 }.
\end{align}
Then, $\operatorname{r}_{\mathcal{K},\regs}^2$ as given in  \cref{eq::max-radius} satisfies for all $A\geq A_\gamma$ that
\begin{align*}
	\mathcal{R} \left( \Delta_{\mathcal{K}\vert\gamma/2}  \vert 	\regDv, \regC ,A^2\operatorname{r}_{\mathcal{K}, \operatorname{s}}^2\eta_{\mathcal{K},\regs}^2 \right) \leq \gamma.
\end{align*}
\end{pr}

\begin{proof}[Proof of \cref{pr::uniform-testing-radius}]
	Under the null hypothesis the claim follows for each $\operatorname{s}\in\mathcal{S}$ from \cref{eq::type-one-error} together with \cref{pr::boundsquantiles-updated} and $\alpha = \gamma /2$, that is,
	\begin{align*}
		\ipM[\denU,\denNX]{n}\left( \Delta_{k\vert\gamma/2}  = 1\right) = \ipM[\denU,\denNX]{n} \left( \eqfk{k}  \geq\tau_{k\vert\gamma/2}\right)\leq \frac{\gamma}{2\vert\mathcal{K}\vert}
	\end{align*}
	for all $k\in\mathcal{K}$, and $\sum_{k\in\mathcal{K}} \frac{\gamma}{2\vert\mathcal{K}\vert} = \gamma/2$. Thus, it holds
	\begin{align*}
		\ipM[\denU,\denNX]{n}\left( \Delta_{\mathcal{K}\vert\gamma/2}  = 1\right) \leq \frac{\gamma}{2}.
	\end{align*}
	If $\denX\in\regDv$ satisfies the separation condition  \cref{eq::separationcond}  for $\tau = \tau_{k^\circ\vert\gamma/2}$ for some $k^\circ\in\mathcal{K}$, then, the assumptions of \cref{pr::boundsquantiles-alternative} are fulfilled with $\beta = \gamma/2$ and it holds with \cref{eq::type-two-error} that  
	\begin{align*}
			\ipM[\denU,\denX]{n}\left( \Delta_{\mathcal{K}\vert\gamma/2}  = 0\right) \leq \min_{k\in\mathcal{K}} \ipM[\denU,\denX]{n}\left( \Delta_{k\vert\gamma/2}  = 0\right) \leq \ipM[\denU,\denX]{n}\left( \Delta_{k^\circ\vert\gamma/2}  = 0\right)\leq  \frac{\gamma}{2}.
	\end{align*}
	Consequently, the claim of the proposition follows, if for every element of the alternative, that is, $\denX\in\regDv$ with $\denX-\denNX\in\regC$ and $\wqf(\denX-\denNX)\geq A^2 \operatorname{r}_{\mathcal{K}, \operatorname{s}}^2\eta_{\mathcal{K},\regs}^2 $ for some $A^2\geq A_\gamma^2$, the the separation condition  \cref{eq::separationcond} is satisfied for $\tau = \tau_{k^\circ\vert\gamma/2}$ and $k^\circ := k_{\mathcal{K}, \operatorname{s}}  \in\mathcal{K}$ as defined in \cref{eq::choice-k}.
	%It sufficient to show that elements of the alternative satisfy the separation condition  \cref{eq::separationcond}  for $\tau_{k^\circ\vert\gamma/2}$ with  $k^\circ := k_{\mathcal{K}, \operatorname{s}} (n) \in\mathcal{K}$ as defined in \cref{eq::choice-k}.
	%Let $\denX-\denNX\in\regC$ satisfy $\wqf(\denX-\denNX)\geq A^2 \operatorname{r}_{\mathcal{K}, \operatorname{s}}^2(n)\cdot \eta_{\mathcal{K},\regs}^2(n) $ for some $A^2\geq A_\gamma^2$.  
	Then, rewriting the separation condition \cref{eq::separationcond}, it remains to show that
	\begin{align}
			\wqfk{k^\circ }(\denX-\denNX)  &\geq 2\tau_{k^\circ\vert\gamma/2} +  2\left(18\cstC[\denU] \icstVz[\denX\vert\denU]{}   
			  + 2\cdot 138986 \frac{L_{\gamma/8}}{\gamma} \frac{\sqrt{2k^\circ}}{n}\right) L_{\gamma/8}^{1/2} \frac{\sqrt{\vBo{k^\circ}}}{n} \nonumber\\
			  &\quad+ 2\cdot 112 \cstC[\denU] \cstV[\denX\vert\denU]\frac{L_{\gamma/8}}{\gamma} \frac{\sqrt{\mBo{k^\circ}}}{n}.\label{eq::separationcond-this}
		\end{align}
		%with $\tau_{k^\circ\vert\gamma/2} $ as in \cref{eq::quantiletau-new}. 
		Since $\denX-\denNX\in\regC$, we have analogously to  \cref{eq::bias-calc1}
		\begin{align*}
			\wqf(\denX - \denNX)-\wqfk{k^\circ }(\denX-\denNX) \leq \frac{\iwF[]{2}(k^\circ )}{\operatorname{s}^2(k^\circ ) } R^2
		\end{align*}
		and, thus, exploiting $\wqf(\denX-\denNX)\geq A_\gamma^2 \operatorname{r}_{\mathcal{K}, \operatorname{s}}^2 \eta_{\mathcal{K},\regs}^2 $ we get
		\begin{align}
			\wqfk{k^\circ }(\denX-\denNX) &\geq \wqf(\denX - \denNX) -  \frac{\iwF[]{2}(k^\circ )}{\operatorname{s}^2(k^\circ )} R^2 \geq A_\gamma^2   \operatorname{r}_{\mathcal{K}, \operatorname{s}}^2\eta_{\mathcal{K},\regs}^2- \frac{\iwF[]{2}(k^\circ )}{\operatorname{s}^2(k^\circ ) } R^2.\label{eq::bias-bound-new}
		\end{align}
		Using $L_{\gamma/4}\geq 1$ it holds
		\begin{align*}
			L_{\gamma/(4\vert \mathcal{K}\vert)} &= 1 - \log(\gamma/(4\vert \mathcal{K}\vert)) = 1- \log (\gamma/4) + \log(\vert \mathcal{K}\vert)\\
			& \leq ( 1- \log (\gamma/4)) (1 + \log(\vert \mathcal{K}\vert)) = L_{\gamma/4} \delta_{\mathcal{K}}^{-2}.
		\end{align*}
		In addition, due to $p\geq 2$ it holds $2^{1/(p-1)}\leq 2$,
		\begin{align*}
			\frac{L_{\gamma/4}^{2} L^{1/2}_{\gamma/4}}{\gamma^{1/(p-1)}}\leq \frac{L_{\gamma/8}^{5/2}}{\gamma}.
		\end{align*}
		Also, we have $  \cstV[\denNX|\denU] \leq \cstV[\denNX|\denU] (p)$. Thus, we obtain for the critical value
		\begin{align}
			 \tau_{k^\circ\vert\gamma/2} &= 	\left( 18 \cstC[\denU]  \cstV[\denNX|\denU] (p) + 2\cdot 69493 \frac{\sqrt{2k^\circ}\vert \mathcal{K}\vert^{1/(p-1)}}{n} \frac{L_{\frac{\gamma}{4\vert\mathcal{K}\vert}}^2}{\gamma} \right)  L_{\frac{\gamma}{4\vert\mathcal{K}\vert}}^{1/2} \frac{\sqrt{\delfour{k^\circ}}}{n}\nonumber\\
			&\qquad + 52 \cstV[\denNX|\denU]   \cstC[\denU]  L_{\frac{\gamma}{4\vert\mathcal{K}\vert}} \frac{\sqrt{\delinf{k^\circ}}}{n}\nonumber\\
			&\leq \left( 18 \frac{L_{\gamma/8}}{\gamma} \cstC[\denU]  \cstV[\denNX|\denU] (p) + 138986 \frac{\sqrt{2k^\circ}\vert \mathcal{K}\vert^{1/(p-1)}}{n \delta_{\mathcal{K}}^4} \frac{L_{\gamma/8}^{5/2}}{\gamma} \right)   \frac{\sqrt{\delfour{k^\circ}}}{\delta_{\mathcal{K}}n}\nonumber\\
			&\qquad + 52 \cstV[\denNX|\denU]   \cstC[\denU]  \frac{L_{\gamma/8}}{\gamma}  \frac{\sqrt{\delinf{k^\circ}}}{n\delta_{\mathcal{K}}^2}\nonumber\\
			&\leq \left( 70\frac{L_{\gamma/8}}{\gamma}\cstC[\denU]  \cstV[\denNX|\denU] (p)  +  138986 \frac{\sqrt{2k^\circ}\vert \mathcal{K}\vert^{1/(p-1)}}{n \delta_{\mathcal{K}}^4} \frac{L_{\gamma/8}^{5/2}}{\gamma}  \right)  \frac{\delta_{\mathcal{K}}^{-1} \sqrt{\delfour{k^\circ}} \vee \delta_{\mathcal{K}}^{-2} \sqrt{\delinf{k^\circ}} }{n}.\label{eq::bound-tko}
		\end{align}
		Inserting in \cref{eq::bias-bound-new} the definition of $A_\gamma$ and using that $\operatorname{r}_{\mathcal{K}, \operatorname{s}} = r_{k^\circ, \operatorname{s}}$ and $ \eta_{\mathcal{K},\regs}^2\geq 1$ we get 
		\begin{align*}
			\wqfk{k^\circ }(\denX-\denNX) &\geq (  R^2 + 140 \frac{L_{\gamma/8}}{\gamma} \cstC[\denU]\cstV[\denNX|\denU](p) + 260 \frac{L_{\gamma/8}}{\gamma} \cstC[\denU]\operatorname{v} + 833934 \frac{L^{5/2}_{\gamma/8}}{\gamma} \operatorname{d}\eta_{\mathcal{K},\regs}^2)\\
			&\qquad \cdot \left( \frac{\iwF[]{2}(k^\circ )}{\operatorname{s}^2(k^\circ ) } \vee \frac{1}{n} \left(\delta_\mathcal{K}^{-1}\sqrt{\delfour{k^\circ}} \vee \delta_\mathcal{K}^{-2}\sqrt{\delinf{k^\circ}}\right) \right) -\frac{\iwF[]{2}(k^\circ )}{\operatorname{s}^2(k^\circ ) } R^2.
		\end{align*}
		Since $\operatorname{v}\geq  \icstVz[\denX|\denU]{}$ for $f\in\regDv$ and  $\frac{\sqrt{2k^\circ} \vert \mathcal{K}\vert^{1/(p-1)}}{\delta_{\mathcal{K}}^4 n}\leq \operatorname{d}  \eta_{\mathcal{K},\regs}^2$ it follows
		\begin{align*}
			\wqfk{k^\circ }(\denX-\denNX) &\geq  \left((2\cdot 70) \frac{L_{\gamma/8}}{\gamma} \cstC[\denU]\cstV[\denNX|\denU](p) + (2\cdot 18 + 2\cdot 112) \frac{L_{\gamma/8}}{\gamma} \icstVz[\denX|\denU]{}\cstC[\denU]\right.\\
			&\qquad \left. +(2+4) 138986\frac{L^{5/2}_{\gamma/8}}{\gamma} \frac{\sqrt{2k^\circ} \vert \mathcal{K}\vert^{1/(p-1)}}{\delta_{\mathcal{K}}^4 n}\right) 
		 	\cdot \frac{1}{n} \left(\delta_\mathcal{K}^{-1}\sqrt{\delfour{k^\circ}} \vee \delta_\mathcal{K}^{-2}\sqrt{\delinf{k^\circ}}\right).
			% \wqfk{k^\circ }(\denX-\denNX)  &\geq 2 \tau_{k^\circ\vert\gamma/2} \\
			%  &\qquad+  2\left(18\cstC[\denU] \icstVz[\denX\vert\denU]{}   
			%   + 2\cdot 138986 \frac{L_{\gamma/8}}{\gamma} \frac{\sqrt{2k^\circ}}{n}\right) L_{\gamma/8}^{1/2} \frac{\sqrt{\vBo{k^\circ}}}{n} \nonumber\\
			%   &\qquad+ 2\cdot 112 \cstC[\denU] \cstV[\denX\vert\denU]\frac{L_{\gamma/8}}{\gamma} \frac{\sqrt{\mBo{k^\circ}}}{n}.
		\end{align*}
		Finally, using  the lower bound of the quantile $ \tau_{k^\circ\vert\gamma/2}$ derived in \cref{eq::bound-tko}, $\frac{\vert\mathcal{K}\vert^{1/(p-1)}}{\delta_{\mathcal{K}}^4}\geq 1$ and $\delta_{\mathcal{K}}\in(0,1)$  we obtain \cref{eq::separationcond-this} and this completes the proof.
		%This shows \cref{eq::separationcond-this} and completes the proof.	
\end{proof}
Let us remark that \cref{pr::uniform-testing-radius} yields $\operatorname{r}_{\mathcal{K}, \operatorname{s}}^2\cdot \ \eta_{\mathcal{K},\regs}^2$ as an upper bound for the radius of testing for any finite sample size $n\geq 2$. Similarly, to \cref{co::optimal-radius-testing} the term $ \eta_{\mathcal{K},\regs}^2(n)$ is a deterioration which is typically asymptotically negligible as $n$ tends to infinity for each $\operatorname{s}\in\mathcal{S}$.
Before establishing in which cases this is fulfilled, we first discuss the rate $\operatorname{r}_{\mathcal{K}, \operatorname{s}}^2(n)$ and the impact of its deviations from $\rho^2_{*,\operatorname{s}}(n) $ as well as choices for the collection $\mathcal{K}$.

\paragraph{Naive choice of $\mathcal{K}$}
For $\operatorname{r}_{\mathcal{K}, \operatorname{s}}^2(n)$, we minimize over the collection $\mathcal{K}$ instead of $\Nz$. Ideally, $\mathcal{K}\subseteq\Nz$ is chosen such that its elements approximate the optimal parameter $k^*_{\operatorname{s}}$ given in \cref{eq::parameter-opt} sufficiently well for all $\operatorname{s}\in\mathcal{S}$. For the examples considered in \cref{sec::uperbound-radius-testing}, we have that $k^*_{\operatorname{s}}\leq n^2$ for $n\in\Nz$ large enough, refer to \cref{tab:risks-general-testing}. Hence, a naive choice is $\mathcal{K}_o := \nset{n^2}$ with $\vert \mathcal{K}_o \vert = n^2$, which yields a factor $\delta_{\mathcal{K}_o}^{-1}= (1 + \log \vert\mathcal{K}_o\vert)^{1/2}$ of order $(\log n)^{1/2}$. Denote here ${k}_o:= k_{\mathcal{K}_o, \operatorname{s}} (n)$ and
\begin{align*}
	V_{\mathcal{K}_o, \operatorname{s}}(n) := \frac{1}{n} \left(\delta_{\mathcal{K}_o}^{-1}\sqrt{\delfour{{k}_o}} \vee \delta_{\mathcal{K}_o}^{-2}\sqrt{\delinf{{k}_o}}\right).
\end{align*}
The resulting rates for the examples presented in \cref{sec::uperbound-radius-testing} are given in \cref{tab:risks-general-choice-K2}.  Further assume from now on that $\parM +a > -\frac14$ for the (o.s.-o.s.) case (first line).  Note that in the example of both densities to be ordinarily smooth (first line) and the case of super smoothness of $\operatorname{s}$ and ordinary smoothness of the error density (third line), the rate worsens by a logarithmic factor, while the case of an ordinary smooth regularity $\operatorname{s}$ and super smoothness of the error density (second line), the rate stays the same as in the non-adaptive case, see \cref{tab:risks-general-testing}. Consequently, we will not consider this case any further in this section. In the fourth column $\kappa\in\pRz$ is some constant depending on $a$ and $\sigma$. However, the exponential term dominates. 
	\begin{table}[!ht]
	\centering
		\begin{tabular}{@{}c c c c c c@{}}
			\toprule
			$\operatorname{s}(t)$&$\vert \Mcg{c}{\denU}(t)  \vert $& $\frac{\iwF{2}(k)}{\operatorname{s}^2(k)}$& $V_{\mathcal{K}_o, \operatorname{s}}(n) $ &$k_o = k_{\mathcal{K}_o , \operatorname{s}}(n)$ &	$\operatorname{r}_{\mathcal{K}_o, \operatorname{s}}^2(n)$\\
			\midrule
			$(1+t^2)^{\frac{s}{2}}$ & $(1+t^2)^{-\frac{\parM}{2}}$ & $k^{-2(s-a)}$ & $\frac{\log n}{n}{k}_o^{2(\parM+a) + \frac{1}{2}}$  &$\left(\frac{n}{\log n}\right)^{\frac{2}{4s+4\parM +1}}$& $\left(\frac{n}{\log n}\right)^{\frac{-4(s-a)}{4s+4\parM +1}}$\\
			$(1+t^2)^{\frac{s}{2}}$  & $\exp( - \vert t\vert^{\sigma})$ & $k^{-2(s-a)}$& $\frac{n}{\log n}{k}_o^{\kappa}\exp(2{k}_o^\sigma)$ &$(\log n)^{\frac{1}{\sigma}}$ & $(\log n)^{-\frac{2(s-a)}{\sigma}}$\\
			$\exp(\vert t\vert^s)$& $(1+t^2)^{-\frac{\parM}{2}}$ & $\exp(2k^s)$ &$\frac{n}{\log n}{k}_o^{2(\parM+a) + \frac{1}{2}}$ & 
			$(\log n)^{\frac{1}{s}}$  &$\frac{\log(n)}{n} (\log n)^{\frac{2(\parM+a)+1/2}{s}}$\\
			\bottomrule
		\end{tabular}
		\caption{Order of the radius of testing for  $\iwF{2}(t)\sim (1+t^2)^{a}$.}
		\label{tab:risks-general-choice-K2}
	\end{table}

\paragraph{Optimized choice of $\mathcal{K}$}
In most cases, a finer grid than  $\mathcal{K}=\nset{n^2}$ approximates the minimization over $\Nz$ well enough. 
More precisely, we consider a minimization over a geometric grid 
\begin{align*}
	\mathcal{K}_g := \{1\}\cup \{2^j, j\in\nset{\log(n^2)}\}. 
\end{align*}
Choosing $\mathcal{K}=\mathcal{K}_g$ results in an adaptive factor $\delta_{\mathcal{K}_g}^{-1}$ of oder $(\log \log n)^{1/2}$. Denote
\begin{align*}
	\operatorname{r}^2_{\mathcal{K}_g, 1} (n) := \arg\min_{k\in\mathcal{K}_g}  \left(\frac{\iwF{2}(k)}{\operatorname{s}^2(k)} \vee \frac{1}{n}\delta_{\mathcal{K}_g}^{-1}\sqrt{\delfour{k}}\right),\\
	\operatorname{r}^2_{\mathcal{K}_g, 2} (n) := \arg\min_{k\in\mathcal{K}_g}  \left(\frac{\iwF{2}(k)}{\operatorname{s}^2(k)} \vee \frac{1}{n}\delta_{\mathcal{K}_g}^{-2}\sqrt{\delinf{k}}\right).
\end{align*}
Then, it holds that $\operatorname{r}^2_{\mathcal{K}_g, \operatorname{s}} (n) = \operatorname{r}^2_{\mathcal{K}_g, 1} (n)\vee \operatorname{r}^2_{\mathcal{K}_g, 2} (n)$, by Lemma A.2.1 in \cite{SchluttenhoferDiss}, since $\frac{\iwF{2}(k)}{\operatorname{s}^2(k)}$ is by \cref{ass::classes-testing}  non-increasing and both $\frac{1}{n}\delta_{\mathcal{K}_g}^{-1}\sqrt{\delfour{k}}$ and $\frac{1}{n}\delta_{\mathcal{K}_g}^{-2}\sqrt{\delinf{k}}$ are non-decreasing in $k$.
\cref{tab:radius-Kg} gives for $\mathcal{K}_g$ the order of both $\operatorname{r}^2_{\mathcal{K}_g, 1}(n)$ and $\operatorname{r}^2_{\mathcal{K}_g, 2} (n) $, for the case that the Mellin transform of the error density is ordinarily smooth and the density $\denX$ ordinarily smooth (first line) or super smooth (second line).

\begin{table}[!ht]
		\centering
		\begin{tabular}{c c c c}
			\toprule
			$\operatorname{s}(t)$&$\vert \Mcg{c}{\denU}(t)  \vert $ &$\operatorname{r}^2_{\mathcal{K}_g, 1} (n)$ & $\operatorname{r}^2_{\mathcal{K}_g, 2} (n)$ \\
			\midrule
			$(1+t^2)^{\frac{s}{2}}$ & $(1+t^2)^{-\frac{\parM}{2}}$  &$\left(\frac{n}{(\log \log n)^{1/2}}\right)^{-\frac{4(s-a)}{4s+4\parM +1}}$  & $\left(\frac{n}{\log \log n}\right)^{-\frac{4(s-a)}{4s+4\parM }}$\\
			$\exp(\vert t\vert^s)$& $(1+t^2)^{-\frac{\parM}{2}}$  & $\frac{(\log \log n)^{1/2}}{n} (\log n)^{\frac{2(\parM+a)+1/2}{s}}$ & $\frac{\log\log n}{n} (\log n)^{\frac{2(\parM+a)}{s}}$\\
			\bottomrule
		\end{tabular}
		\caption{Order of the radius of testing $\iwF{2}(t)\sim (1+t^2)^{a}$.}
		\label{tab:radius-Kg}
\end{table}

Note, that $\operatorname{r}^2_{\mathcal{K}_g, 2} (n)$ is asymptotically negligible compared with $\operatorname{r}^2_{\mathcal{K}_g, 1} (n)$. Hence, the upper bound in \cref{pr::uniform-testing-radius} asymptotically reduces to $\operatorname{r}^2_{\mathcal{K}_g, 1} (n)$ for both the first and the third line.

If the unknown density is super smooth and the error density is ordinary smooth, the smaller geometric collection 
 \begin{align*}
	 \mathcal{K}_m = \{1\} \cup  \{2^j, j\in\nset{m^{-1}\log \log n}\},
 \end{align*}
 for $m>0$ is still sufficient. Then, the adaptive factor becomes of order 
$(\log \log \log n)^{1/2}$. The corresponding rates are given in \cref{tab:radius-Km}.  Again, the term $\operatorname{r}^2_{\mathcal{K}_m, 2} (n)$ is asymptotically negligible compared with $\operatorname{r}^2_{\mathcal{K}_m, 1} (n)$ and the upper bound in \cref{pr::uniform-testing-radius} asymptotically reduces to $\operatorname{r}^2_{\mathcal{K}_m, 1} (n)$. 

\begin{table}[!ht]
		\centering
		\begin{tabular}{c c c c }
			\toprule
			$\operatorname{s}(t)$&$\vert \Mcg{c}{\denU}(t)  \vert $ &$\operatorname{r}^2_{\mathcal{K}_m, 1} (n)$ & $\operatorname{r}^2_{\mathcal{K}_m, 2} (n)$ \\
			\midrule
			$\exp(\vert t\vert^s)$& $(1+t^2)^{-\frac{\parM}{2}}$ &$\frac{(\log \log \log n)^{1/2}}{n}(\log n)^{\frac{2(\parM+a)+1/2}{s}}$  &$\frac{\log \log \log n}{n}(\log n)^{\frac{2(\parM +a)}{s}}$\\
			\bottomrule
		\end{tabular}
		\caption{Order of the radius of testing $\iwF{2}(t)\sim (1+t^2)^{a}$.}
		\label{tab:radius-Km}
\end{table}

For more details on the derivation of the rates in \cref{tab:radius-Kg,tab:radius-Km}, we refer to the calculations of Illustration 4.3.6. in \cite{SchluttenhoferDiss}, where analogous rates are calculated for a circular convolution model.

\paragraph{Order of factor $\eta_{\mathcal{K}\vert n}^2$}
Let us discuss under which assumptions the factor $\eta_{\mathcal{K}\vert n}^2$ is of order one for the given examples such that $\operatorname{r}^2_{\mathcal{K}, \operatorname{s}} (n)$ describes indeed the order of the radius of testing. First note that in all discussed collections it holds $\vert\mathcal{K}\vert \leq n^2$ and thus $\delta_{\mathcal{K}}^{-4}\leq (\log n)^2$. Let us start with 
the case that the Mellin transform of the error density is ordinarily smooth and the density $\denX$ super smooth (second line of \cref{tab:risks-general-choice-K2}). In this case, it holds that
\begin{align*}
	 \frac{\sqrt{2k_{\mathcal{K}, \operatorname{s}} (n)}\vert \mathcal{K}\vert^{1/(p-1)}}{ \operatorname{d}n \delta_{\mathcal{K}}^4 } & \lesssim \frac{n^{\frac{1}{4s+4\parM +1}}n^{\frac{2}{(p-1)}}}{n (\log n)^{-2}}.
\end{align*}
This is of order less or equal to one whenever 
\begin{align*}
	\frac{2}{p-1} -1 + \frac{1}{4s+4\parM +1} <0.
\end{align*}
For example, if $\frac14 < \parM + s$ and $p\geq 5$, this is satisfied. Consequently, if the regularity parameters $\sigma$ and $s$ are sufficiently large and the moment condition \cref{ass::error-p-moments} is satisfied for $\denU$ and $\denNX$, then the parameter $\eta_{\mathcal{K}\vert n}^2$ is of order one.
For the other two examples, line one and three of \cref{tab:risks-general-choice-K2}, we have that $\sqrt{k_{\mathcal{K}, \operatorname{s}}(n)}$ is of logarithmic order in $n$, and, consequently, $p\geq 5$ is sufficient for $\eta_{\mathcal{K}\vert n}^2\sim 1$ to hold, as well. That is, the radius of testing is of order as given in \cref{tab:risks-general-choice-K2,tab:radius-Kg,tab:radius-Km}, depending on the choice of $\mathcal{K}$.

\section{Simulation study}\label{sec::sim-testing}
In this section, we illustrate the behaviour of the test $\Delta_{k,\alpha}$ presented in \cref{eq::test-stat2} and the data-driven max-test $\Delta_{\mathcal{K}\vert\alpha} $ defined in \cref{eq::max-test-updated}.

To do so we consider as null hypothesis the case of an super-smooth null hypothesis density and ordinarily smooth error density $\denU$. Specifically, we set the null hypothesis to be log-normally distributed with parameters $\mu  =0, \sigma^2=1$, i.e., for $x\in\pRz$
\begin{align*} 
    \denNX(x)&= \frac{1}{\sqrt{2\pi x^2}} \exp(-\log(x)^2/2)
\end{align*}
and $U$ being Pareto-distributed, i.e.
\begin{align*}
      \denU(x)&= \mathds{1}_{(1, \infty)}(x)x^{-2}.
\end{align*}
For the random variable $X$ we look both at the case that the null hypothesis is satisfied, i.e., $\denX = \denNX$, and, at the case where the density of $X$ belongs to the alternative. More precisely, we consider for $X$ the two cases
\begin{align*}
    \denX_1(x)&= \denNX(x),\\
    \denX_2(x)&=2x\mathds{1}_{(0,1)}(x).
\end{align*}
Setting $c=0.5$, the corresponding Mellin-transforms are given by
\begin{align*}
    \Mcg{\frac12}{\denX_1}(t)&= 2/(1.5 + 2\pi it),\\
    \Mcg{\frac12}{\denX_2}(t)&= \exp\left((-0.5+2\pi i t)^2 /2\right),\\
    \Mcg{\frac12}{\denU}(t)&=(1.5-2\pi it)^{-1}.
\end{align*}
We refer again to \cite{BrennerMiguelDiss} for more details on Mellin transforms. Let $\wF(t)=1$ for $t\in\Rz$. The separation of the density to the null hypothesis is obviously zero for $\denX_1$ , that is $\wqf(\denX_1-\denNX) = 0$. For the second density $\denX_2$ we have $\wqf(\denX_2-\denNX) = 0.5$.

Below, we show the results of our simulation study in Figure \ref{fig::testing-null} for the first example, $\denX = \denX_1$, and in Figure \ref{fig::testing-alternative} for second example, $\denX = \denX_2$. The boxplots of the values of the test statistic $\eqfk{k}$ are depicted over 50 iterations for $k\in\{0.5,0.6,\ldots,1.4\}$ and $n\in\{100,500\}$.
The examples were chosen such that the densities and their Mellin transforms have a simple representation. However, the bias is for $k \leq 1.4 $ already too small to see variation in the estimates. Minimizing over a subset of $\Nz$, we would always choose $k=1$ for the given sample sizes. Only immensely higher sample sizes would lead to a different output. Thus, we rescaled the set of dimension parameters on the interval $[0.5,1.4]$. 

In both Figures  \ref{fig::testing-null} and \ref{fig::testing-alternative}, the triangles indicate different estimates for $\alpha$-quantiles for $\alpha =0.1$. The dark blue triangle indicates the quantile $\tau_k(0.1)$ proposed in \cref{eq::test-stat2} and the light blue one the corresponding quantile $\tau_{k\vert 0.1}$ for the max-test defined in \cref{eq::quantiletau-new} for $\mathcal{K}= \{0.5,0.6,\ldots,1.4\}$. Note that since $L_{\gamma/(4\vert \mathcal{K}\vert)} \leq L_{\gamma/4} \delta_{\mathcal{K}}^{-2}$ (see proof of \cref{pr::uniform-testing-radius}), these theoretical quantiles differ by $\delta_{\mathcal{K}}^{-1} = (1 + \log(10))^{1/2}\approx 1.82$ up to constants. In dark brown, the empirical $0.1$-quantiles under the null hypothesis are indicated. 
The light brown triangles give the empirical quantiles multiplied by the factor $\delta_{\mathcal{K}}^{-1}$ to give an empirical quantile for the max-test. The numbers in brackets in the legend give the rejection rate of the corresponding max-tests for each example.

As already mentioned,  $n= \{100,500\}$ are relatively small sample sizes and the bias is relatively small compared to the variance. Consequently, we decide to use a value equal to $0.6$ for the constants of both  $\tau_k(0.1)$ and $\tau_{k\vert 0.1}$. The effect of this can be seen in the boundaries of the interval considered for parameter $k$: Under the null hypothesis in Figure \ref{fig::testing-null} we note that the test does not adhere to the significance level of $0.1$ at the boundaries. Choosing the theoretical constants instead, the test would always reject for these examples. 

However, close to $k=1$ the theoretical quantiles appear to match the empirical quantiles. Looking at the alternative in Figure \ref{fig::testing-alternative}, the separation appears to be large enough such that the test has high power already for a sample size of $n=500$.

  \begin{figure}[!ht]
    \centering
    \begin{subfigure}[b]{1\textwidth}
        \centering
        \includegraphics[width=\textwidth]{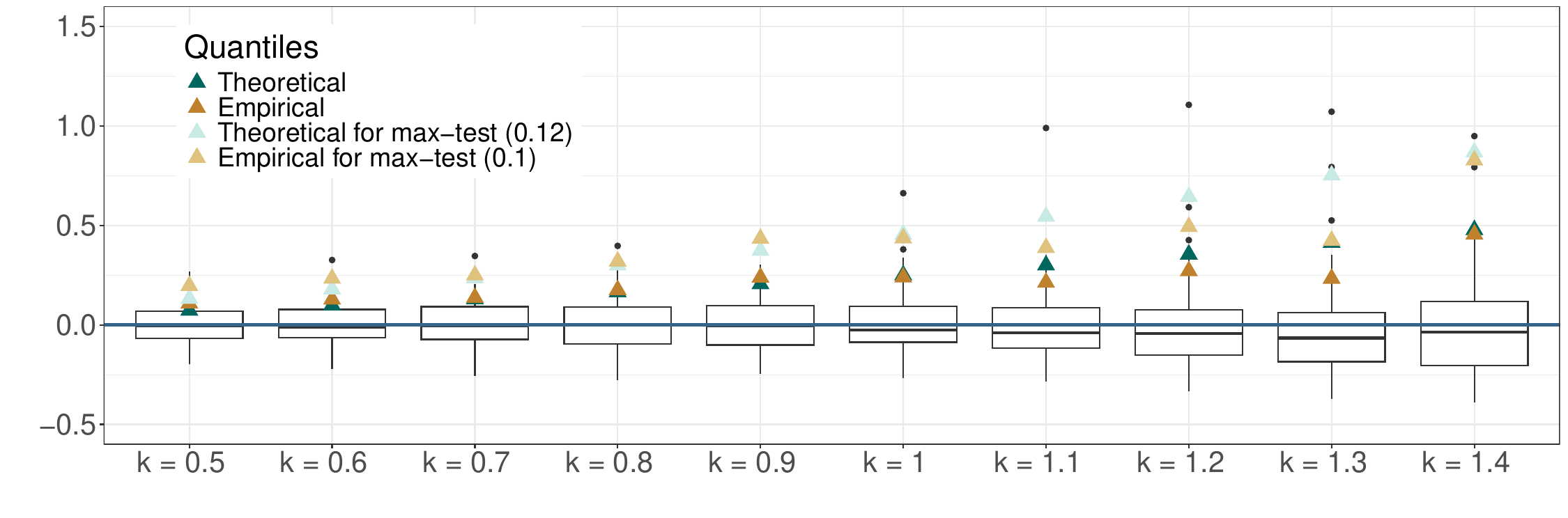}
        \caption{Example 1, $n=100$}
        \label{fig:sub1-test}
    \end{subfigure}
    \begin{subfigure}[b]{1\textwidth}
        \centering
        \includegraphics[width=\textwidth]{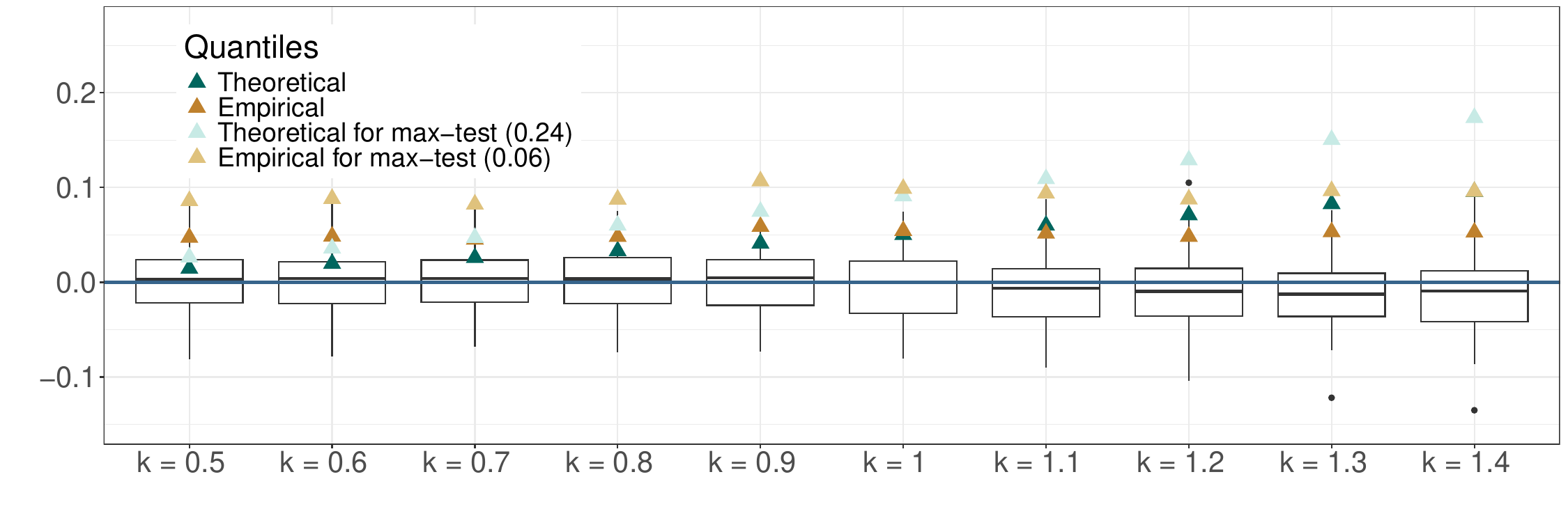}
        \caption{Example 1, $n=500$}
        \label{fig:sub2-test}
    \end{subfigure}\\
    \caption{The boxplots represent the values of $\eqfk{k}$ for $\denX = \denX_1$ over 50 iterations. The horizontal lines indicate $\wqf = 0$. The triangles indicate different estimates for $0.1$-quantiles. The numbers given in the legend indicate the rejection rate of the corresponding test.}
    \label{fig::testing-null}
\end{figure}

\begin{figure}[!ht]
   
    \centering
    \begin{subfigure}[b]{1\textwidth}
        \centering
        \includegraphics[width=\textwidth]{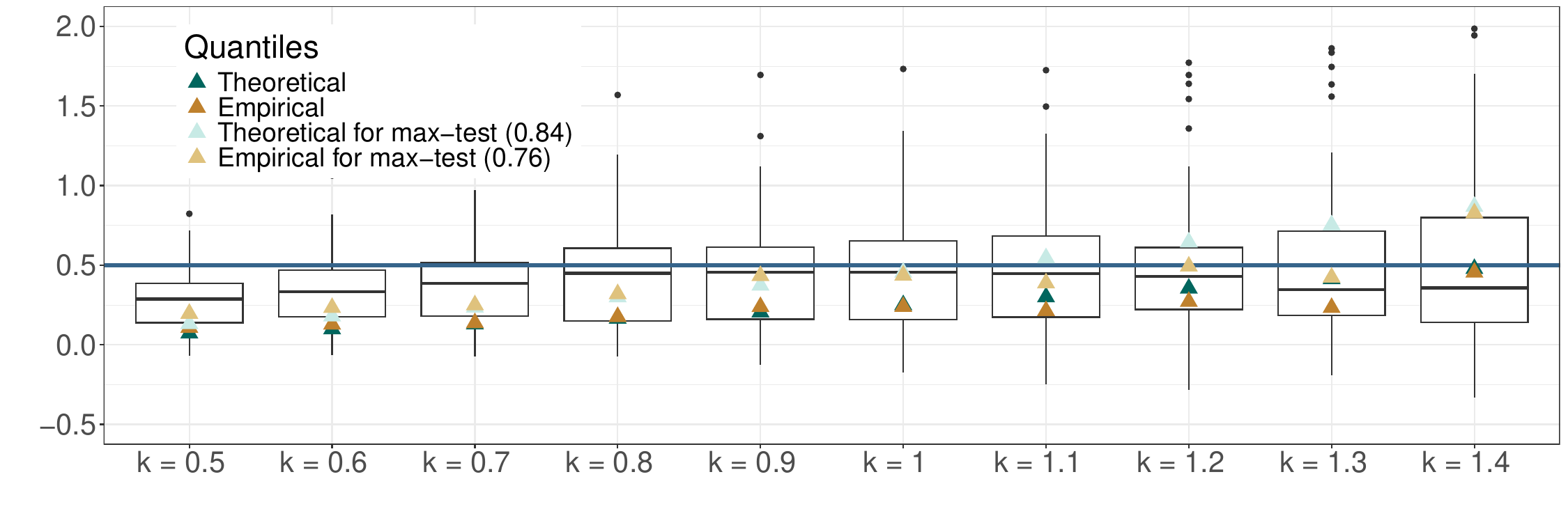}
        \caption{Example 2, $n=100$}
        \label{fig:sub1-test2}
    \end{subfigure}
    \begin{subfigure}[b]{1\textwidth}
        \centering
        \includegraphics[width=\textwidth]{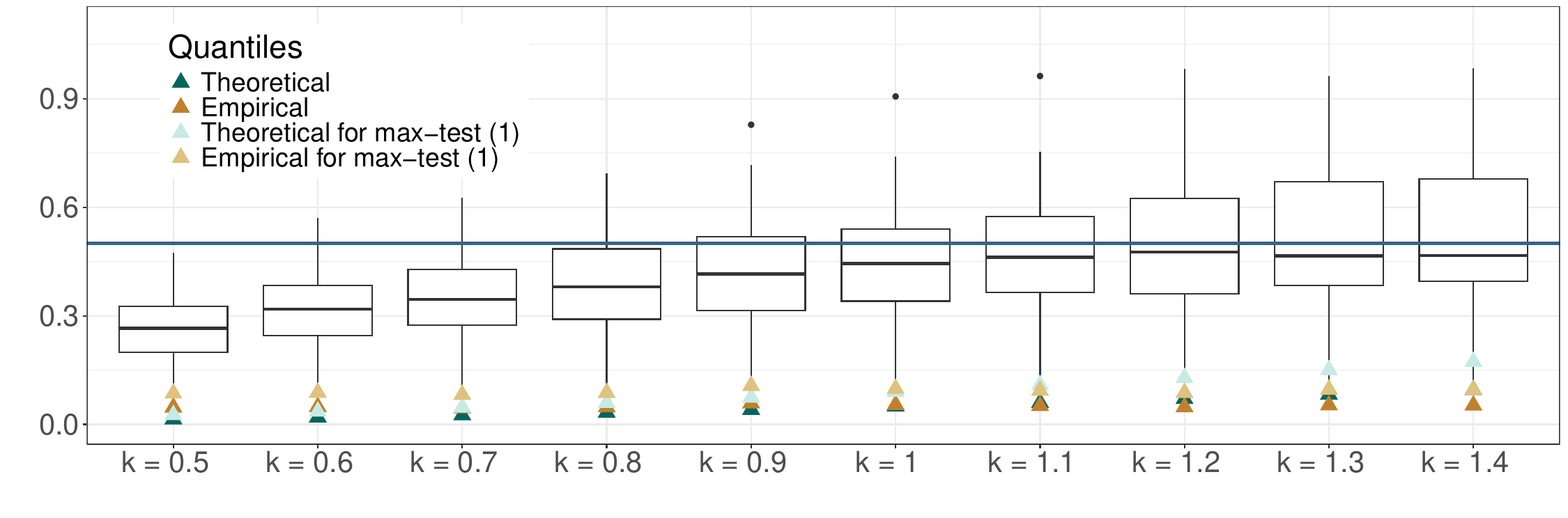}
        \caption{Example 2, $n=500$}
        \label{fig:sub2-test2}
    \end{subfigure}\\
    \caption{The boxplots represent the values of $\eqfk{k}$ for $\denX = \denX_2$ over 50 iterations. The horizontal lines indicate $\wqf = 0.5$. The triangles indicate different estimates for $0.1$-quantiles. The numbers given in the legend indicate the rejection rate of the corresponding test.}
    \label{fig::testing-alternative}
\end{figure}

\pagebreak

\paragraph{Acknowlegdegments} This work is funded by Deutsche Forschungsgemeinschaft (DFG, German Research Foundation) under Germany’s Excellence Strategy EXC-2181/1-39090098 (the Heidelberg STRUCTURES Cluster of Excellence).

\paragraph{Disclosure statement} The authors report there are no competing interests to declare.

% --------------------------------------------------------------------
% <<BibFile>>
% --------------------------------------------------------------------
\bibliography{AITFMC_bib} 

%--------------------------------------------------------------------
% <<Appendix>>
% --------------------------------------------------------------------
\appendix
\setcounter{subsection}{0}
\section*{Appendix}
\numberwithin{equation}{subsection}  
\renewcommand{\thesubsection}{\Alph{subsection}}
\renewcommand{\theco}{\Alph{subsection}.\arabic{co}}
\numberwithin{co}{subsection}
\renewcommand{\thelem}{\Alph{subsection}.\arabic{lem}}
\numberwithin{lem}{subsection}
\renewcommand{\therem}{\Alph{subsection}.\arabic{rem}}
\numberwithin{rem}{subsection}
\renewcommand{\thepr}{\Alph{subsection}.\arabic{pr}}
\numberwithin{pr}{subsection}
\subsection{Auxiliary results}\label{a::auxiliary}

In this section, we collect results from the literature used in the proofs of this work.
We start by reviewing some properties of U-statistics. These results hold for random variables on any measurable space, in particular for  $(\pRz, \psB)$.

For $(X_j)_{j\in\llbracket n \rrbracket}$, $n\in\Nz$ independent and identically distributed random variables in  $(\pRz, \psB)$ and symmetric kernel $h\colon \pRz \times \pRz \rightarrow \Rz$, a  U-statistic is a sum of the form
\begin{align*}
	U_n = \frac1{n(n-1)}\sum_{\substack{j_1\not= j_2\\ j_1,j_2 \in\llbracket n\rrbracket }}  h(X_{j_1}, X_{j_2}).
\end{align*}
The U-statistic is called canonical if for all $i,j\in\llbracket n\rrbracket$ and all $x,y\in \pRz$
\begin{align*}
	\E\left[h(X_i, y)\right] = 	\E\left[h(x, X_j)\right] = 0. 
\end{align*}

The next assertion, a concentration inequality for canonical U-statistics, provides our key argument in order to control the deviation of the test statistics.  It is a reformulation of Theorem 3.4.8 in \cite{Gine2016}. 
We use the notation $\Vert \cdot \Vert_{\Lp[\infty]{+,2}}$ to indicate the essential supremum for functions on $\pRz \times \pRz$. 

\begin{pr}\label{pr::Ustatconcentration}
	Let $U_n$ be a canonical U-statistic for $(X_j)_{j\in\llbracket n \rrbracket}$, $n\geq 2$ i.i.d. $\pRz$-valued random variables and kernel $h\colon \pRz\times \pRz \rightarrow \Rz$ bounded and symmetric. Let
	\begin{align*}
		&A\geq \Vert h\Vert_{\Lp[\infty]{+,2}}\\
		&B^2\geq \Vert \E[h^2(X_1, \cdot)] \Vert_{\Lp[\infty]{+}}\\
		&C^2 \geq\E[h^2(X_1,X_2)]\\
		&D \geq  \sup \{\E[h(X_1,X_2)\xi(X_1)\zeta(X_2)]: \E[\xi^2(X_1)]\leq 1, \E[\zeta^2(X_2)]\leq 1\}.
	\end{align*}
	Then, for all $x\geq 0$
	\begin{align*}
		\ipM[]{}\left(U_n \geq 8 \frac Cn x^{1/2} + 13 \frac Dn x + 261\frac{B}{n^{3/2}}x^{\frac32} + 343 \frac{A}{n^2} x^2\right) \leq \exp(1- x).
	\end{align*}
\end{pr}
We bound the variance of an U-statistic by the following adaption of Theorem 3, Chapter 1,  in \cite{Lee1990}.

\begin{lem}\label{lem::var-bound-testing}
	Let $U_n$ be a U-statistic for $(X_j)_{j\in\llbracket n \rrbracket}$, $n\geq 2$ i.i.d. $\pRz$-valued random variables. Then
	\begin{align*}
		\operatorname{Var}(U_n) \leq \frac{2}{n(n-1)}\E^2[\vert h(X_1,X_2)\vert^2].
	\end{align*}
\end{lem}

Now, we turn to  norm inequalities for multiplicative convolution. The next result can be found in \cite{Comte2025}.
\begin{lem}\label{lem::normineq}
		\begin{itemize}
			\item[(i)] Let $h_1, h_2\in \Lp[1]{+}(\basMSy{2c-2})$ for $c\in\Rz$. Then $h_1 * h_2 \in \Lp[1]{+}(\basMSy{2c-2})$ with $\Vert h_1 * h_2 \Vert_{{\Lp[1]{+}(\basMSy{2c-2})}} \leq \Vert h_1 \Vert_{{\Lp[1]{+}(\basMSy{2c-2})}} \Vert h_2 \Vert_{{\Lp[1]{+}(\basMSy{2c-2})}} $. If $h_1$ and $h_2$ are probability densities, equality holds.
			\item[(ii)] Let $h_1, h_2\in \Lp[1]{+}(\basMSy{c-1})$ and $h_2\in \Lp[2]{+}(\basMSy{2c-1})$ for  $c\in\Rz$. Then $h_1 * h_2 \in  \Lp[2]{+}(\basMSy{2c-1})$ with $\Vert h_1 * h_2 \Vert_{{\Lp[2]{+}(\basMSy{2c-1})}} \leq \Vert h_1 \Vert_{{\Lp[1]{+}(\basMSy{c-1})}} \Vert h_2 \Vert_{{\Lp[2]{+}(\basMSy{2c-1})}} $.
			\item[(iii)] Let $h_1, h_2\in \Lp[\infty]{+}(\basMSy{2c-1})$ and $h_2 \in \Lp[1]{+}(\basMSy{2c-2})$ for  $c\in\Rz$. Then $h_1 * h_2 \in \Lp[\infty]{+}(\basMSy{2c-1})$ with $\Vert h_1 * h_2 \Vert_{{\Lp[\infty]{+}(\basMSy{2c-1})}} \leq \Vert h_1 \Vert_{{\Lp[\infty]{+}(\basMSy{2c-1})}} \Vert h_2 \Vert_{{\Lp[1]{+}(\basMSy{2c-2})}} $.
			\item[(iv)] If $f\in\Lp[1]{+}({\basMSy{2c-2}})$ and $\varphi\in\Lp[\infty]{+}( \basMSy{2c-1})\cap \Lp[1]{+}({\basMSy{2c-2}})$ are probability densities and $ g := f* \varphi$, then we have
	\begin{align}\label{eq::mm3}
		\Vert g \Vert_{\Lp[\infty]{+}( \basMSy{2c-1})} \leq \Vert  f \Vert_{\Lp[1]{+}({\basMSy{2c-2}})} \Vert \varphi \Vert_{\Lp[\infty]{+}( \basMSy{2c-1})}  =  \frac{\Vert \varphi \Vert_{\Lp[\infty]{+}( \basMSy{2c-1})}}{\Vert  \varphi \Vert_{\Lp[1]{+}({\basMSy{2c-2}})}} \Vert  g \Vert_{\Lp[1]{+}({\basMSy{2c-2}})}.
	\end{align}
		\end{itemize}
\end{lem}
	Next, we recall properties of the Mellin transform, which we use in the sequel without further reference. See \cite{Comte2025} for the proof of the following result.
	\begin{lem}\label{lem::Mellinprops}
	\begin{itemize}
	\item[(i)] It holds for any $h\in\Lp[1]{+}(\basMSy{c-1})\cap \Lp[2]{+}(\basMSy{2c-1})$ that
	\begin{align*}
		\overline{\Mcg{c}{h}}(t)	= \Mcg{c}{h}(-t) , \quad \forall t\in\Rz
	\end{align*}
	and, analogously, for $H\in\Lp[1]{}\cap\Lp[2]{}$ it holds that 
	\begin{align*}
		\overline{\Mcgi{c}{H}}(x) = \Mcgi{c}{H(-\cdot)}(x), \quad \forall x\in\pRz.
	\end{align*}
	\item[(ii)] Let $g\in\Lp[\infty]{+}(\basMSy{2c-1})$, then it holds for $\mathds{1}_{[-k,k]}h\in \Lp[2]{} $ that 
	\begin{align*}
		\Ex[g] \left[ \left\vert\int_{-k}^k Y^{c-1+ 2\pi it} h(t)  dt \right\vert^2 \right] \leq  \Vert g \Vert_{\Lp[\infty]{+}( \basMSy{2c-1})} \Vert  \mathds{1}_{[-k,k]} h \Vert_{\Lp[2]{}}^2
	\end{align*}
	\item[(iii)] For $g\in\Lp[\infty]{+}(\basMSy{2c-1})$ and $h\in\Lp[2]{+}(g)$ we have that
	\begin{align*}
		\int_{\Rz} \left\vert \Ex[g] \left[ Y^{c-1+2\pi it} h(Y) \right]\right\vert^2 dt \leq \Vert g \Vert_{\Lp[\infty]{+}( \basMSy{2c-1})}  \Ex[g] [h^2(Y)].
	\end{align*}
	\end{itemize}
\end{lem}
\subsection{Proofs for \cref{sec::bounds-quantiles}}\label{app::u-stat}

The key element to analyse the behaviour of the test statistic $\etqfk{k} $ defined in \cref{eq::estimatorq2kb} is the  decomposition \cref{eq::decompositiontest}. It holds
\begin{align*}
	\etqfk{k}  = U_{k} + 2W_{k} + \wqfk{k}(\denX-\denNX) 
\end{align*}
with the canonical U-statistic 
\begin{align}\label{eq:canonicalustat-app}
	U_{k} :=  \frac{1}{n(n-1)}\sum_{\substack{j_1\not= j_2\\ j_1,j_2 \in\llbracket n\rrbracket }} h_k(Y_{j_1}, Y_{j_2})
\end{align}
where for $x,y\in\pRz$
\begin{align}\label{eq::hk-ustat}
	h_k(x,y) :=\int_\Rz \indc{-k}{k} \vert\Mcg{c}{\denU}\vert^{-2}  \cfun{x}{\cdot} \cfunn{y}{\cdot} \iwF{2} d\Lm
\end{align}
and the centred linear statistic
\begin{align}\label{eq::wk-app}
	W_{k} := \frac{1}{n} \sum_{j \in\llbracket n \rrbracket}  \int_{-k}^k   \frac{\cfun{Y_j}{t} (\overline{\Mcg{c}{\denY}}(t) - \overline{\Mcg{c}{\denNY}}(t))}{\vert\Mcg{c}{\denU}(t)\vert^{2}}\iwF{2}(t) d\Lm(t).
\end{align}
Due to this decomposition, in \cref{sec:u-stat}, we first collect auxiliary results for the U-statistic $U_{k}$. In \cref{sec::lin-stat-results}, we continue with results on the linear statistic $W_k$.

\subsubsection{\textit{U-statistic results}}\label{sec:u-stat}
In order to obtain bounds of the quantiles of $U_k$ given in \cref{eq:canonicalustat-app}, we apply the concentration inequality for canonical U-statistics given in \cref{pr::Ustatconcentration}.
Note that for $x\in\pRz$  the term $\vert x^{c-1+2\pi i t}\vert = x^{c-1}$ and, hence, the kernel $h_k$ given in \cref{eq::hk-ustat} are in general not bounded. Therefore, we decompose $h_k$ in a bounded and a remaining unbounded part. More precisely, given $\delta\in\pRz$ specified later, we denote for $y\in\pRz$ and $t\in\Rz$
\begin{align*}
	\bcut{k}(y,t):= \mathds{1}_{[0,\delta]}(y^{c-1}) y^{c-1+2\pi it} \ \text{ and } \ \ucut{k}(y,t ):= \mathds{1}_{(\delta,\infty)}(y^{c-1}) y^{c-1+2\pi it}.
\end{align*}
Define the bounded part of kernel $h_k$ as
\begin{align}\label{eq::bounded-h-kernel-testing}
	h^b_k(x, y) := \int_{-k}^{k} \frac{(\bcut{k}(y,t)-\Ex[\denY][\bcut{k}(Y_1,t)])(\bcut{k}(x,-t)-\Ex[\denY][\bcut{k}(Y_1,-t)])}{\vert \Mcg{c}{\denU}(t)\vert^2} \iwF{2}(t) d\Lm(t).
\end{align}
	Then, $h^b_k$ is indeed bounded as can be seen in the proof of \cref{lem::UstatconcentrationMellinCase}. Analogously, define
	\begin{align}
		h^u_k(x, y) &:= \int_{-k}^{k} 		\frac{(\ucut{k}(y,t)-\Ex[\denY][\ucut{k}(Y_1,t)])(\ucut{k}(x,-t)-\Ex[\denY][\ucut{k}(Y_1,-t)])}{\vert \Mcg{c}{\denU}(t)\vert^2} \iwF{2}(t) d\Lm(t)\nonumber\\
	 	& \quad + \int_{-k}^{k} 	\frac{(\ucut{k}(y,t)-\Ex[\denY][\ucut{k}(Y_1,t)])(\bcut{k}(x,-t)-\Ex[\denY][\bcut{k}(Y_1,-t)])}{\vert \Mcg{c}{\denU}(t)\vert^2} \iwF{2}(t) d\Lm(t) \nonumber \\
	 	& \quad + \int_{-k}^{k} 				\frac{(\bcut{k}(y,t)-\Ex[\denY][\bcut{k}(Y_1,t)])(\ucut{k}(x,-t)-\Ex[\denY][\ucut{k}(Y_1,-t)])}{\vert \Mcg{c}{\denU}(t)\vert^2} \iwF{2}(t) d\Lm(t). \label{eq::mixed-h-kernel}
	\end{align}
 	The kernels $h^b_k$ and $h^{u}_k$ are also symmetric and real-valued and we have $h_k= h^b_k + h^u_k$. Denote by $U^b_k$ and $U^{u}_k$ the corresponding canonical U-statistics, respectively, where $U_k = U^b_k + U^{u}_k$ evidently. 
	We first derive a quantile-bound for the bounded part $U^b_k$ in \cref{lem::quantil-bounded-ustat} using \cref{pr::Ustatconcentration}  and, subsequently, a quantile-bound for the unbounded part $U^{u}_k$ in \cref{lem::quantil-unbounded-ustat}. Both results are combined to derive a bound for the quantile of $U_k$ in \cref{lem::quantil-ustat}.
	We make use of the quantile-bounds of $U_k$ both under the null and the alternative. However, only under the null we impose eventually a stronger moment condition, that is,  $\denNX\in\Lp[1]{+}({\basMSy{2p(c-1)}})$ for some $p\geq 2$, while under the alternative $\denX\in\Lp[1]{+}({\basMSy{4(c-1)}})$ is sufficient.
	Therefore, we derive the following quantile-bounds of $U_k$ assuming $g:=\denU\mc\denX$ with $\denU,\denX\in\Lp[1]{+}({\basMSy{2p(c-1)}})$ for some $p\geq2$.
	We introduce  the following notation
	\begin{align*}
		\cstV[\denX|\denU](p) = (\Vert \denX \mc \denU\Vert_{\Lp[1]{+}({\basMSy{2p(c-1)}})} \vee 1).
	\end{align*}
	Note that $ \cstV[\denX|\denU] = \cstV[\denX|\denU](1)$.
	For the bounded U-statistic, the following lemma gives the constants for applying \cref{pr::Ustatconcentration}. It was shown in \cite{Comte2025} and we state the proof here for the sake of completeness.

	 \begin{lem}[Constants for the bounded U-statistic]\label{lem::UstatconcentrationMellinCase}
	Let the assumptions of \cref{lem::quantil-ustat} be satisfied. The kernel $h^b_k$ introduced in \cref{eq::bounded-h-kernel-testing} with $\delta\in\pRz$ is real-valued, bounded, symmetric and satisfies for all $y\in\pRz$ that $\Ex[\denY][h_k^b(Y, y)]= 0$.  Moreover,
	\begin{align*}
		A &:= 4 \delta^2 \Vert \mathds{1}_{[-k,k]} / \Mcg{c}{\denU} \Vert_{\Lp[2]{}(\iwF{2})}^2,\\
		B^2 &:= 4\cstC[\denU] \cstV[\denX|\denU]\delta^2 \Vert \mathds{1}_{[-k,k]} / \Mcg{c}{\denU} \Vert_{\Lp[4]{}(\iwF{4})}^4,\\
		C^2 &:=\cstC[\denU] \icstV[\denX|\denU]{2}\Vert \mathds{1}_{[-k,k]} / \Mcg{c}{\denU} \Vert_{\Lp[4]{}(\iwF{4})}^4, \\
		D &:= 4\cstC[\denU] \cstV[\denX|\denU] \Vert \mathds{1}_{[-k,k]} / \Mcg{c}{\denU} \Vert_{\Lp[\infty]{}(\wf)}^2
	\end{align*}
	satisfy the conditions of the U-statistic concentration inequality given in \cref{pr::Ustatconcentration}. 
\end{lem}

\begin{proof}[Proof of \cref{lem::UstatconcentrationMellinCase}]
	We compute the quantities $A$, $B$, $C$ and $D$ verifying the four inequalities of \cref{pr::Ustatconcentration}. Consider $A$. Since $\vert \bcut{k}(y,t) - \Ex[\denY][\bcut{k}(Y_1, t)]\vert \leq 2\delta$ for all $t\in\Rz$ and $y\in\pRz$ we have
	\begin{align*}
		\sup_{x,y\in\pRz} \vert h^b_k(x,y) \vert \leq 4 \delta^2 \int_{-k}^k \frac{1}{\vert\Mcg{c}{\denU}(t)\vert^2} \iwF{2}(t) d\Lm(t)  = A.
	\end{align*}
	Consider $B$. Using that the U-statistic is canonical and \cref{lem::Mellinprops} (iii) we get for $y\in\pRz$ 
	\begin{align}
		&\Ex[\denY] [\vert h^b_k(y, Y_1)\vert^2] \leq \Ex[\denY] \left[\left\vert\int_{-k}^k  	\bcut{k}(Y_1,-t)\frac{(\bcut{k}(y,t)-\Ex[\denY][\bcut{k}(Y_1,t)])}{\vert \Mcg{c}{\denU}(t)\vert^2} \iwF{2}(t) d\Lm(t) \right\vert^2 \right]\nonumber\\
		&=  \Ex[\denY] \left[\mathds{1}_{[0, \delta]}(Y_1^{c-1}) \left\vert\int_{-k}^k Y_1^{c-1-2\pi it} \frac{(\bcut{k}(y,t)-\Ex[\denY][\bcut{k}(Y_1,t)])}{\vert \Mcg{c}{\denU}(t)\vert^2} \iwF{2}(t) d\Lm(t) \right\vert^2 \right] \nonumber\\
		&\leq \Ex[\denY] \left[ \left\vert\int_{-k}^k Y_1^{c-1-2\pi it} \frac{(\bcut{k}(y,t)-\Ex[\denY][\bcut{k}(Y_1,t)])}{\vert \Mcg{c}{\denU}(t)\vert^2} \iwF{2}(t) d\Lm(t) ö\right\vert^2 \right]\nonumber \\
		&\leq \Vert \denY  \Vert_{\Lp[\infty]{+}(\basMSy{2c-1})}  \int_{-k}^k \frac{\vert\bcut{k}(y,t)-\Ex[\denY][\bcut{k}(Y_1,t)]\vert^2}{\vert \Mcg{c}{\denU}(t)\vert^4} \iwF{4}(t) d\Lm(t).  \label{eq::step-B-constant}
	\end{align}
	Consequently, $\vert \bcut{k}(y,t) - \Ex[\denY][\bcut{k}(Y_1, t)]\vert \leq 2\delta$ for all $t\in\Rz$  and $y\in\pRz$  implies
	\begin{align*}
		\sup_{y\in\pRz} 	\Ex[\denY] [\vert h^b_k(y, Y_1)\vert^2] &\leq \cstC[\denU] \cstV[\denX|\denU] 4\delta^2  \int_{-k}^k \frac{1}{\vert \Mcg{c}{\denU}(t)\vert^4} \iwF{4}(t) d\Lm(t) = B^2. 
	\end{align*}
	Consider C. Keep in mind that for all $t\in\Rz$ 
	\begin{align*}
		&\Ex[\denY] \left[\vert \bcut{k}(Y_1,t) - \Ex[\denY][\bcut{k}(Y_1,t)]\vert^2\right] \leq \Ex[\denY]\left[\vert \bcut{k}(Y_1,t) \vert^2\right] \\
		&=  \Ex[\denY]\left[\mathds{1}_{[0,\delta]}(Y_1^{c-1})\vert Y_1^{c-1+2\pi it} \vert^2\right] \leq \Ex[\denY] [Y_1^{2c-2}] \leq \cstV[\denX|\denU].
	\end{align*}
	Using additionally the calculations of  \cref{eq::step-B-constant} for $B$, it follows
	\begin{align*}
		\iEx[\denY]{2} \left[ \vert h^{b}_k(Y_1,Y_2)\vert^2  \right] &\leq  \Vert \denY  \Vert_{\Lp[\infty]{+}(\basMSy{2c-1})}  \int_{-k}^k \Ex[\denY][\vert\bcut{k}(y,t)-\Ex[\denY][\bcut{k}(Y_1,t)]\vert^2]\frac{\iwF{4}(t)}{\vert \Mcg{c}{\denU}(t)\vert^4}  d\Lm(t)\\
		&\leq \cstC[\denU] \icstV[\denX|\denU]{2} \int_{-k}^k \frac{1}{\vert \Mcg{c}{\denU}(t)\vert^4} \iwF{4}(t) d\Lm(t) = C^2.
	\end{align*}
	Finally, consider $D$. $D=C$ satisfies the condition, but we find a sharper bound. For this, we first see that for all $\xi\in\Lp[2]{+}(\denY)$ due to  \cref{lem::Mellinprops} (iii) we have
	\begin{align*}
		\int_{\Rz} \left\vert \Ex[\denY] \left[\xi(Y_1)\bcut{k}(Y_1, t)\right] \right\vert^2 d\Lm(t) &= 
		\int_{\Rz} \left\vert \Ex[\denY] \left[Y_1^{c-1+2\pi it}\xi(Y_1)\mathds{1}_{[0, \delta]}(Y_1^{c-1}) \right] \right\vert^2 d\Lm(t)\\
		&\leq \Vert \denY  \Vert_{\Lp[\infty]{+}(\basMSy{2c-1})} \Ex[\denY] \left[\vert \xi(Y_1)\vert^2 \mathds{1}_{[0, \delta]}(Y_1^{c-1})\right]\\
		&\leq  \cstC[\denU] \cstV[\denX|\denU] \Ex[\denY] [\vert \xi(Y_1)\vert^2].
	\end{align*}
	Similarly, by \cref{lem::Mellinprops} (iii) we get 
	\begin{align*}
	\int_{\Rz} \vert \Ex[\denY] [Y_1^{c-1+2\pi it}\mathds{1}_{[0, \delta]}(Y_1^{c-1}) ] \vert^2 d\Lm(t) \leq \cstC[\denU] \cstV[\denX|\denU],
	\end{align*}
	which in turn implies
	\begin{align*}
		&\int_{\Rz} \vert \Ex[\denY]\left[\xi(Y_1)\Ex[\denY][\bcut{k}(Y_1, t)]\right] \vert^2 d\Lm(t) \leq \int_{\Rz} \Ex[\denY]\left[\vert \xi(Y_1)\vert^2\right]  \vert\Ex[\denY][\bcut{k}(Y_1, t)]\vert^2 d\Lm(t)\\
		&=   \Ex[\denY]\left[\vert \xi(Y_1)\vert^2\right] \int_{\Rz} \vert \Ex[\denY] [Y_1^{c-1+2\pi it}\mathds{1}_{[0, \delta]}(Y_1^{c-1}) ] \vert^2 d\Lm(t)\leq \Ex[\denY]\left[\vert \xi(Y_1)\vert^2\right] \cstC[\denU] \cstV[\denX|\denU].
	\end{align*}
	Combining the last bounds it follows
	 \begin{align*}
		\int_{\Rz} \vert \Ex[\denY][\xi(Y_1)(\bcut{k}(Y_1, t)-\Ex[\denY][\bcut{k}(Y_1,t)])]\vert^2 d\Lm(t)\leq 4\cstC[\denU] \cstV[\denX|\denU]\Ex[\denY][\vert\xi(Y_1)\vert^2].
	 \end{align*}
	  Consequently, for all $\xi,\zeta\in\Lp[2]{+}(\denY)$ with $\Ex[\denY][\vert \xi(Y_1)\vert^2]\leq 1$ and $\Ex[\denY][\vert \zeta(Y_1)\vert^2]\leq 1$ it follows
	\begin{align*}
		&\int_{\pRz} h^b_k(x,y)\xi(x)\zeta(y) \pM[\denY](dx)\pM[\denY](dy)\\
		& = \int_{-k}^k \text{\small$\frac{\Ex[\denY][\xi(Y_1)(\bcut{k}(Y_1,t)-\Ex[\denY][\bcut{k}(Y_1,t)])]\Ex[\denY][\zeta(Y_2)(\bcut{k}(Y_2,-t)-\Ex[\denY][\bcut{k}(Y_1, -t)])]}{\vert \Mcg{c}{\denU}(t)\vert^2}$} \iwF{2}(t)d\Lm(t)\\
		&\leq \sup_{\substack{\xi \in \Lp[2]{+}(\denY),\\ \Ex[\denY][\vert \xi(Y_1)\vert^2]\leq 1 }}\left\{ \int_{-k}^k \frac{\vert\Ex[\denY][\xi(Y_1)(\bcut{k}(Y_1,t)-\Ex[\denY][\bcut{k}(Y_1,t)])]\vert^2}{\vert \Mcg{c}{\denU}(t)\vert^2}\iwF{2}(t)d\Lm(t)\right\}\\
		&\leq \Vert \mathds{1}_{[-k,k]} / \Mcg{c}{\denU} \Vert_{\Lp[\infty]{}(\wF)}^2 \sup_{\substack{\xi \in \Lp[2]{+}(\denY),\\ \Ex[\denY][\vert \xi(Y_1)\vert^2]\leq 1 }}\left\{ \int_{-k}^k \vert\Ex[\denY][\xi(Y_1)(\bcut{k}(Y_1,t)-\Ex[\denY][\bcut{k}(Y_1,t)])]\vert^2 d\Lm(t)\right\}\\
		&\leq 4\cstC[\denU] \cstV[\denX|\denU] \Vert \mathds{1}_{[-k,k]} / \Mcg{c}{\denU} \Vert_{\Lp[\infty]{}(\wF)}^2 = D.
	\end{align*}
	This concludes the proof.
\end{proof}

With this result we show the following lemma for the bounded part of the U-statistic \cref{eq::canonicalustat}. 
\begin{lem}[Quantile of the bounded U-statistic]\label{lem::quantil-bounded-ustat}
	Let the assumptions of \cref{lem::quantil-ustat} be satisfied. Consider for each $n\geq 2$ and any $\delta\in\pRz$ the canonical U-statistic $U_k^b$. For $\gamma\in(0,1)$ and 
	\begin{align*}
		\tau_k^b(\gamma)&:= \left( 9 \cstC[\denU] \cstV[\denX|\denU] + 69493 \frac{\delta^2}{n} \sqrt{2k}L_\gamma^2\right)L_\gamma^{1/2} \frac{\Vert \mathds{1}_{[-k,k]} / \Mcg{c}{\denU} \Vert_{\Lp[4]{}(\iwF{4})}^2 }{n} \\
		&\qquad\qquad+ 52 \cstC[\denU] \cstV[\denX|\denU]  L_\gamma \frac{\Vert \mathds{1}_{[-k,k]} / \Mcg{c}{\denU} \Vert_{\Lp[\infty]{}(\wf)}^2 }{n}
	\end{align*}
	we have 
	\begin{align*}
		\ipM[\denU, \denX]{n} (U_k^b \geq \tau_k^b(\gamma)) \leq \gamma.
	\end{align*}
\end{lem}

\begin{proof}[Proof of \cref{lem::quantil-bounded-ustat}]
	We intend to apply the exponential inequality for canonical U-statistics given in \cref{pr::Ustatconcentration} using quantities $A$, $B$, $C$, and $D$ which verify the four required inequalities. Precisely, for $\gamma\in(0,1)$ and $x = L_\gamma = 1- \log\gamma$ due to \cref{pr::Ustatconcentration} we have 
	\begin{align*}
		\ipM[\denU, \denX]{n} \left(U_{k} \geq 8\frac Cn  L_\gamma^{1/2} + 13 \frac Dn L_\gamma + 261\frac{B}{n^{3/2}}L_\gamma^{3/2} + 343 \frac{A}{n^2} L_\gamma^2 \right)  \leq \exp(1-L_\gamma)= \gamma.
   \end{align*}
   Exploiting \cref{lem::quantil-bounded-ustat} we have that 
   \begin{align*}
   &8\frac Cn  L_\gamma^{1/2} + 13 \frac Dn L_\gamma + 261\frac{B}{n^{3/2}}L_\gamma^{3/2} + 343 \frac{A}{n^2} L_\gamma^2\\
   &= 8\frac{\icstC[\denU]{1/2} \icstV[\denX|\denU]{}\Vert \mathds{1}_{[-k,k]} / \Mcg{c}{\denU} \Vert_{\Lp[4]{}(\iwF{4})}^2}{n}  L_\gamma^{1/2} + 52 \frac{\cstC[\denU] \cstV[\denX|\denU] \Vert \mathds{1}_{[-k,k]} / \Mcg{c}{\denU} \Vert_{\Lp[\infty]{}(\wf)}^2}{n} L_\gamma \\
   &\quad + 522 \frac{\icstC[\denU]{1/2} \icstV[\denX|\denU]{1/2}\delta \Vert \mathds{1}_{[-k,k]} / \Mcg{c}{\denU} \Vert_{\Lp[4]{}(\iwF{4})}^2}{n^{3/2}}L_\gamma^{3/2} + 1372 \frac{\delta^2 \Vert \mathds{1}_{[-k,k]} / \Mcg{c}{\denU} \Vert_{\Lp[2]{}(\iwF{2})}^2}{n^2} L_\gamma^2
   \end{align*}
   which together with $\Vert \mathds{1}_{[-k,k]} / \Mcg{c}{\denU} \Vert_{\Lp[2]{}(\iwF{2})}^4 \leq (2k)\Vert \mathds{1}_{[-k,k]} / \Mcg{c}{\denU} \Vert_{\Lp[4]{}(\iwF{4})}^4 $ implies
   \begin{align*}
		&8\frac Cn  L_\gamma^{1/2} + 13 \frac Dn L_\gamma + 261\frac{B}{n^{3/2}}L_\gamma^{3/2} + 343 \frac{A}{n^2} L_\gamma^2\\
		&\leq 52 \frac{\cstC[\denU] \cstV[\denX|\denU] \Vert \mathds{1}_{[-k,k]} / \Mcg{c}{\denU} \Vert_{\Lp[\infty]{}(\wf)}^2}{n} L_\gamma +  \frac{\Vert \mathds{1}_{[-k,k]} / \Mcg{c}{\denU} \Vert_{\Lp[4]{}(\iwF{4})}^2}{n} L_\gamma^{1/2} \\
		&\qquad \qquad \cdot \left( 8 \icstC[\denU]{1/2} \cstV[\denX|\denU] + 522 \frac{\icstC[\denU]{1/2} \icstV[\denX|\denU]{1/2}\delta }{n^{1/2}}L_\gamma + 1372 \frac{\delta^2 \sqrt{2k}}{n} L_\gamma^{3/2} \right).
   \end{align*}
   Using that $2ab\leq a^2 + b^2$ for $a,b\in\pRz$, we have that 
   \begin{align*}
	522 \frac{\icstC[\denU]{1/2} \icstV[\denX|\denU]{1/2}\delta }{n^{1/2}}L_\gamma \leq \cstC[\denU]\cstV[\denX|\denU] + (522/2)^2 \frac{\delta^2L_\gamma^2}{n}
   \end{align*}
   and, thus, since $\sqrt{2k}, L_\gamma \geq 1$
   \begin{align*}
	&8\frac Cn  L_\gamma^{1/2} + 13 \frac Dn L_\gamma + 261\frac{B}{n^{3/2}}L_\gamma^{3/2} + 343 \frac{A}{n^2} L_\gamma^2\\
	&\leq  52 \frac{\cstC[\denU] \cstV[\denX|\denU] \Vert \mathds{1}_{[-k,k]} / \Mcg{c}{\denU} \Vert_{\Lp[\infty]{}(\wf)}^2}{n} L_\gamma +  \frac{\Vert \mathds{1}_{[-k,k]} / \Mcg{c}{\denU} \Vert_{\Lp[4]{}(\iwF{4})}^2}{n} L_\gamma^{1/2} \\
	&\qquad \qquad \cdot \left( 9 \cstC[\denU] \cstV[\denX|\denU] + ((522/2)^2 +  1372)\frac{\delta^2 \sqrt{2k}}{n} L_\gamma^{2} \right)  = \tau_k^b(\gamma)
   \end{align*}
   with $(522/22)^2 +1372 = 69493$ due to \cref{pr::Ustatconcentration} follows
   \begin{align*}
	\ipM[\denU, \denX]{n} (U_k^b \geq \tau_k^b(\gamma)) \leq \gamma
   \end{align*}
   which completes the proof.
\end{proof}

Further, we give a quantile-bound of the unbounded part $U_k^u$ of the U-statistic given in \cref{eq::canonicalustat}.
\begin{lem}[Quantile of the unbounded U-statistic]\label{lem::quantil-unbounded-ustat}
	Let the assumptions of \cref{lem::quantil-ustat} be satisfied. Consider for each $n\geq 2$ and $\delta\in\pRz$ the canonical U-statistic $U_k^u$. For $\gamma\in(0,1)$ setting
	\begin{align*}
		\tau_k^u (\gamma) := \frac{\delta^{(1-p)}}{\gamma^{1/2}} 6\icstC[\denU]{1/2} \icstV[\denX|\denU]{1/2} (\cstV[\denX|\denU](p))^{1/2} \frac{\Vert \mathds{1}_{[-k,k]} / \Mcg{c}{\denU} \Vert_{\Lp[4]{}(\iwF{4})}^2}{n}
  	\end{align*}
	we have 
	\begin{align*}
		\ipM[\denU, \denX]{n}(U_k^u \geq \tau_k^u(\gamma)) \leq \gamma.
	\end{align*}
\end{lem}

\begin{proof}[Proof of \cref{lem::quantil-unbounded-ustat}]
	We start the proof with the observation that the symmetric and real-valued kernel $h_k^u$ given in \cref{eq::mixed-h-kernel} satisfies 
	\begin{align}
		& \iEx[\denY]{2} [\vert h^u_k (Y_1, Y_2)\vert^2]\nonumber\\
		&\leq 3  \iEx[\denY]{2} \left[\left\vert\int_{-k}^{k} 		\frac{(\ucut{k}(Y_2,t)-\Ex[\denY][\ucut{k}(Y,t)])(\ucut{k}(Y_1,-t)-\Ex[\denY][\ucut{k}(Y,-t)])}{\vert \Mcg{c}{\denU}(t)\vert^2} \iwF{2}(t) d\Lm(t)\right\vert^2\right]\nonumber\\
		& \quad + 3  \iEx[\denY]{2} \left[\left\vert \int_{-k}^{k} 	\frac{(\ucut{k}(Y_2,t)-\Ex[\denY][\ucut{k}(Y,t)])(\bcut{k}(Y_2,-t)-\Ex[\denY][\bcut{k}(Y,-t)])}{\vert \Mcg{c}{\denU}(t)\vert^2} \iwF{2}(t) d\Lm(t) \right\vert^2\right]\nonumber \\
	 	& \quad +  3  \iEx[\denY]{2} \left[\left\vert\int_{-k}^{k} 				\frac{(\bcut{k}(Y_2,t)-\Ex[\denY][\bcut{k}(Y,t)])(\ucut{k}(Y_1,-t)-\Ex[\denY][\ucut{k}(Y,-t)])}{\vert \Mcg{c}{\denU}(t)\vert^2} \iwF{2}(t) d\Lm(t)\right\vert^2\right]\nonumber\\
		\intertext{Note that the second and the third summand are the same due to symmetry of the integrals. Thus, we obtain}
		& \iEx[\denY]{2} [\vert h^u_k (Y_1, Y_2)\vert^2]\nonumber\\
		&\leq 3  \iEx[\denY]{2} \left[\left\vert\int_{-k}^{k} 		\frac{(\ucut{k}(Y_2,t)-\Ex[\denY][\ucut{k}(Y,t)])(\ucut{k}(Y_1,-t)-\Ex[\denY][\ucut{k}(Y,-t)])}{\vert \Mcg{c}{\denU}(t)\vert^2} \iwF{2}(t) d\Lm(t)\right\vert^2\right]\nonumber\\
		& \quad + 6  \iEx[\denY]{2} \left[\left\vert \int_{-k}^{k} 	\frac{(\ucut{k}(Y_2,t)-\Ex[\denY][\ucut{k}(Y,t)])(\bcut{k}(Y_1,-t)-\Ex[\denY][\bcut{k}(Y,-t)])}{\vert \Mcg{c}{\denU}(t)\vert^2} \iwF{2}(t) d\Lm(t) \right\vert^2\right]. \label{eq::some-sum-decomp}
	\end{align}
	Consider the first summand. Due to the independence of $Y_1$ and $Y_2$ we first consider the expectation with respect to $Y_2$ for fixed $Y_1=y\in\pRz$. We have that 
	\begin{align*}
		& \Ex[\denY] \left[\left\vert\int_{-k}^{k} 		\frac{(\ucut{k}(Y_2,t)-\Ex[\denY][\ucut{k}(Y,t)])(\ucut{k}(y,-t)-\Ex[\denY][\ucut{k}(Y,-t)])}{\vert \Mcg{c}{\denU}(t)\vert^2} \iwF{2}(t) d\Lm(t)\right\vert^2\right] \\
		&\leq  \Ex[\denY] \left[\left\vert\int_{-k}^{k} 		\frac{\ucut{k}(Y_2,t)(\ucut{k}(y,-t)-\Ex[\denY][\ucut{k}(Y,-t)])}{\vert \Mcg{c}{\denU}(t)\vert^2} \iwF{2}(t) d\Lm(t)\right\vert^2\right] \\
		&=  \Ex[\denY] \left[ \mathds{1}_{(\delta,\infty)}(Y_2^{c-1}) \left\vert\int_{-k}^{k} 		\frac{ Y_2^{c-1+2\pi it}(\ucut{k}(y,-t)-\Ex[\denY][\ucut{k}(Y,-t)])}{\vert \Mcg{c}{\denU}(t)\vert^2} \iwF{2}(t) d\Lm(t)\right\vert^2\right]\\
		&\leq  \Ex[\denY] \left[  \left\vert\int_{-k}^{k} 		\frac{ Y_2^{c-1+2\pi it}(\ucut{k}(y,-t)-\Ex[\denY][\ucut{k}(Y,-t)])}{\vert \Mcg{c}{\denU}(t)\vert^2} \iwF{2}(t) d\Lm(t)\right\vert^2\right].\\
		\intertext{	Using \cref{lem::Mellinprops} (ii) with $h(t):= 	\frac{(\ucut{k}(y,-t)-\Ex[\denY][\ucut{k}(Y,-t)])}{\vert \Mcg{c}{\denU}(t)\vert^2} \iwF{2} (t)$, it follows for $y\in\pRz$}
		& \Ex[\denY] \left[  \left\vert\int_{-k}^{k} 		\frac{ Y_2^{c-1+2\pi it}(\ucut{k}(y,-t)-\Ex[\denY][\ucut{k}(Y,-t)])}{\vert \Mcg{c}{\denU}(t)\vert^2} \iwF{2}(t) d\Lm(t)\right\vert^2\right]\\
		&\leq  \Vert \denY \Vert_{\Lp[\infty]{+}( \basMSy{2c-1})} \Vert  \mathds{1}_{[-k,k]} \frac{(\ucut{k}(y,\cdot)-\Ex[\denY][\ucut{k}(Y,\cdot)])}{\vert \Mcg{c}{\denU}\vert^2} \Vert_{\Lp[2]{}(\iwF{2})}^2.
	\end{align*}
	Similarly, for the second summand of \cref{eq::some-sum-decomp} we get for $y\in\pRz$ that 
	\begin{align*}
		& \Ex[\denY] \left[\left\vert \int_{-k}^{k} 	\frac{(\ucut{k}(y,t)-\Ex[\denY][\ucut{k}(Y,t)])(\bcut{k}(Y_2,-t)-\Ex[\denY][\bcut{k}(Y,-t)])}{\vert \Mcg{c}{\denU}(t)\vert^2} \iwF{2}(t) d\Lm(t) \right\vert^2\right] \\
		&\leq  \iEx[\denY]{2} \left[\left\vert \int_{-k}^{k} 	\frac{\bcut{k}(Y_2,t)(\ucut{k}(y,-t)-\Ex[\denY][\ucut{k}(Y,-t)])}{\vert \Mcg{c}{\denU}(t)\vert^2} \iwF{2}(t) d\Lm(t) \right\vert^2\right] \\
		&=  \Ex[\denY] \left[\mathds{1}_{[0,\delta]}(Y_2^{c-1}) \left\vert \int_{-k}^{k} 	\frac{ Y_2^{c-1+2\pi it}(\ucut{k}(y,-t)-\Ex[\denY][\ucut{k}(Y,-t)])}{\vert \Mcg{c}{\denU}(t)\vert^2} \iwF{2}(t) d\Lm(t) \right\vert^2\right]\\
		&\leq  \Ex[\denY] \left[  \left\vert\int_{-k}^{k} 		\frac{ Y_2^{c-1+2\pi it}(\ucut{k}(y,-t)-\Ex[\denY][\ucut{k}(Y,-t)])}{\vert \Mcg{c}{\denU}(t)\vert^2} \iwF{2}(t) d\Lm(t)\right\vert^2\right].\\
		\intertext{Using again \cref{lem::Mellinprops} (ii) with $h(t):= 	\frac{(\ucut{k}(y,-t)-\Ex[\denY][\ucut{k}(Y,-t)])}{\vert \Mcg{c}{\denU}(t)\vert^2} \iwF{2}(t)$, it follows for $y\in\pRz$}
		& \Ex[\denY] \left[  \left\vert\int_{-k}^{k} 		\frac{ Y_2^{c-1+2\pi it}(\ucut{k}(y,-t)-\Ex[\denY][\ucut{k}(Y,-t)])}{\vert \Mcg{c}{\denU}(t)\vert^2} \iwF{2}(t) d\Lm(t)\right\vert^2\right]\\
		&\leq  \Vert \denY \Vert_{\Lp[\infty]{+}( \basMSy{2c-1})} \Vert  \mathds{1}_{[-k,k]} \frac{(\ucut{k}(y,\cdot)-\Ex[\denY][\ucut{k}(Y,\cdot)])}{\vert \Mcg{c}{\denU}\vert^2} \Vert_{\Lp[2]{}(\iwF{4})}^2.
	\end{align*}
	Combining the last upper bounds and  \cref{eq::some-sum-decomp}, we obtain 
	\begin{align}
		 \iEx[\denY]{2} [\vert h^u_k (Y_1, Y_2)\vert^2] 
		&\leq 9\Vert \denY \Vert_{\Lp[\infty]{+}( \basMSy{2c-1})} \Ex[\denY] \left[   \Vert  \mathds{1}_{[-k,k]} \frac{(\ucut{k}(Y_1,\cdot)-\Ex[\denY][\ucut{k}(Y,\cdot)])}{\vert \Mcg{c}{\denU}\vert^2} \Vert_{\Lp[2]{}(\iwF{4})}^2\right].\label{eq::upper-boundb6}
	\end{align}
	Further, we have that $\Vert \denY\Vert_{\Lp[\infty]{+}(\basMSy{2c-1})}\leq \cstV[\denX|\denU]   \cstC[\denU]$  and for $p\geq2$ it holds
	\begin{align*}
		 \Ex[\denY] [\vert \mathds{1}_{(\delta,\infty)}(Y_1^{c-1}) Y_1^{c-1+2\pi it}\vert^2]\leq \delta^{2(1-p)} \cstV[\denX|\denU](p)
	\end{align*}
	which in turn implies
	\begin{align*}
		\Ex[\denY] \left[ \left\vert \ucut{k}(Y_1,t)-\Ex[\denY][\ucut{k}(Y,t)] \right\vert^2 \right] \leq \delta^{2(1-p)} \cstV[\denX|\denU](p).
	\end{align*}
	Combining these estimates with \cref{eq::upper-boundb6} we get
	\begin{align*}
		 \iEx[\denY]{2} [\vert h^u_k (Y_1, Y_2)\vert^2]
		&\leq 9  \cstV[\denX|\denU]  \cstC[\denU] \int_{-k}^k \Ex[\denY] \left[ \left\vert \frac{(\ucut{k}(Y_1,t)-\Ex[\denY][\ucut{k}(Y,t)])}{\vert \Mcg{c}{\denU}(t)\vert^2} \right\vert^2 \right] \iwF{4}(t)d\Lm(t)\\
		&\leq 9  \cstV[\denX|\denU]  \cstC[\denU] \delta^{2(1-p)} \cstV[\denX|\denU](p)  \Vert \mathds{1}_{[-k,k]} / \Mcg{c}{\denU} \Vert_{\Lp[4]{}(\iwF{4})}^4.
	\end{align*}
	Applying \cref{lem::var-bound-testing} we immediately obtain
	\begin{align*}
		\iEx[\denY]{n} \left[\vert U_k^u \vert^2\right] &= \frac{2}{n(n-1)}  \iEx[\denY]{2} [\vert h^u_k (Y_1, Y_2)\vert^2]  \leq \frac{36}{n^2}\cstV[\denX|\denU]  \cstC[\denU] \delta^{2(1-p)} \cstV[\denX|\denU](p)  \Vert \mathds{1}_{[-k,k]} / \Mcg{c}{\denU} \Vert_{\Lp[4]{}(\iwF{4})}^4\\
		&= \gamma \vert \tau_k^u(\gamma)\vert^2.
	\end{align*}
	Hence, by applying Markov's inequality the result follows, i.e.,
	\begin{align*}
		\ipM[\denU, \denX]{n}(U_k^u \geq \tau_k^u(\gamma)) \leq \frac{\iEx[\denY]{n} \left[\vert U_k^u \vert^2\right]}{\vert \tau_k^u(\gamma)\vert^2} \leq \gamma,
	\end{align*} 
	which completes the proof.
\end{proof}

\begin{lem}[Quantile of the U-statistic]\label{lem::quantil-ustat}
	Let \cref{ass:well-definedness-testing,ass::error-well-defined,ass::error-p-moments} be satisfied. Consider for each $n\geq 2$ the canonical U-statistic $U_k$ defined in \cref{eq:canonicalustat-app}. For any $\denX\in \Lp[1]{+}(\basMSy{2p(c-1)})$ with $p\geq 2$ and setting
	\begin{align}
		\tau^{U_k}(\gamma) &:= \left( 18 \cstV[\denX|\denU](p) \cstC[\denU] + 69493   \frac{\sqrt{2k}}{n \gamma^{1/(p-1)}} L_{\gamma/2}^{2 - 1/(p-1)}  \right)\frac{\Vert \mathds{1}_{[-k,k]} / \Mcg{c}{\denU} \Vert_{\Lp[4]{}(\iwF{4})}^2}{n} L_{\gamma/2}^{1/2}\nonumber\\
		&\quad + 52 \cstC[\denU] \cstV[\denX|\denU] \frac{\Vert \mathds{1}_{[-k,k]} / \Mcg{c}{\denU} \Vert_{\Lp[\infty]{}(\wf)}^2}{n} L_{\gamma/2},\label{eq::quantile-ustat}
	\end{align}
	for  $\gamma\in(0,1)$, we have
	\begin{align*}
		\ipM[\denU, \denX]{n} (U_k \geq \tau^{U_k}(\gamma)) \leq \gamma.
	\end{align*}
\end{lem}

\begin{proof}[Proof of \cref{lem::quantil-ustat}]
	We start the proof by decomposing 
	the  U-statistic $U_k = U_k^b + U_k^u$ with kernels given in \cref{eq::bounded-h-kernel-testing} and \cref{eq::mixed-h-kernel}
	using the threshold
	\begin{align}\label{eq::choice-delta}
		\delta^2 = L_{\gamma/2}^{1/(1-p)} \gamma^{1/(1-p)} > 0.
	\end{align}
	 We show below that the critical values $\tau^{U_k}(\gamma)$ given in \cref{eq::quantile-ustat},  $\tau_k^b(\gamma)$ and $\tau_k^u(\gamma)$ defined in \cref{lem::quantil-bounded-ustat,lem::quantil-unbounded-ustat}, respectively, satisfy
	\begin{align}\label{eq::quantile-ineq}
		\tau^{U_k}(\gamma) \geq \tau_k^b(\gamma/2) + \tau_k^u(\gamma/2).
	\end{align}
	Consequently, from \cref{lem::quantil-bounded-ustat,lem::quantil-unbounded-ustat} we immediately obtain the result, that is,
	\begin{align*}
		\ipM[\denU, \denX]{n} (U_k \geq \tau^{U_k}(\gamma)) \leq \ipM[\denU, \denX]{n} (U_k^b \geq \tau^{b}_k(\gamma/2)) + \ipM[\denU, \denX]{n} (U_k^u \geq \tau^u_k(\gamma/2))\leq \gamma/2 + \gamma/2 = \gamma.
	\end{align*}
	It remains to show \cref{eq::quantile-ineq}. We have that
	\begin{align*}
		\tau_k^b(\gamma/2) + \tau_k^u(\gamma/2) &= \left( 9 \cstC[\denU] \cstV[\denX|\denU] + 69493 \frac{\delta^2}{n} \sqrt{2k}L_{\gamma/2}^2\right)L_{\gamma/2}^{1/2} \frac{\Vert \mathds{1}_{[-k,k]} / \Mcg{c}{\denU} \Vert_{\Lp[4]{}(\iwF{4})}^2 }{n} \\
		&\qquad\qquad+ 52 \cstC[\denU] \cstV[\denX|\denU]  L_{\gamma/2} \frac{\Vert \mathds{1}_{[-k,k]} / \Mcg{c}{\denU} \Vert_{\Lp[\infty]{}(\wf)}^2 }{n} \\
		&\qquad\qquad + \sqrt{2}\frac{\delta^{(1-p)}}{\gamma^{1/2}} 6\icstC[\denU]{1/2} \icstV[\denX|\denU]{1/2} (\cstV[\denX|\denU](p))^{1/2} \frac{\Vert \mathds{1}_{[-k,k]} / \Mcg{c}{\denU} \Vert_{\Lp[4]{}(\iwF{4})}^2}{n}.\\
		\intertext{Using $\cstV[\denX|\denU], \icstV[\denX | \denU]{2}\leq \cstV[\denX|\denU](p)$ yields}
			\tau_k^b(\gamma/2) + \tau_k^u(\gamma/2) &\leq \left( 9 \cstC[\denU] \cstV[\denX|\denU](p) + 69493 \frac{\delta^2}{n} \sqrt{2k}L_{\gamma/2}^2 + L_{\gamma/2}^{-1/2} 9\frac{\delta^{(1-p)}}{\gamma^{1/2}} \cstC[\denU]\cstV[\denX|\denU](p)  \right)  \\
		&\qquad\qquad \cdot L_{\gamma/2}^{1/2} \frac{\Vert \mathds{1}_{[-k,k]} / \Mcg{c}{\denU} \Vert_{\Lp[4]{}(\iwF{4})}^2 }{n}\\
		&\qquad\qquad+ 52 \cstC[\denU] \cstV[\denX|\denU]  L_{\gamma/2} \frac{\Vert \mathds{1}_{[-k,k]} / \Mcg{c}{\denU} \Vert_{\Lp[\infty]{}(\wf)}^2 }{n}.
	\end{align*}
	With the choice of $\delta$ given in \cref{eq::choice-delta}, it holds $\delta^{1-p} = L_{\gamma/2}^{1/2} \gamma^{1/2}$ and it follows
	\begin{align*}
		&\tau_k^b(\gamma/2) + \tau_k^u(\gamma/2)\\
		 &\leq \left( 18 \cstC[\denU] \cstV[\denX|\denU](p) + 69493 \gamma^{-1/(p-1)}\frac{L_{\gamma/2}^{2- 1/(p-1)} }{n} \sqrt{2k} \right)  L_{\gamma/2}^{1/2} \frac{\Vert \mathds{1}_{[-k,k]} / \Mcg{c}{\denU} \Vert_{\Lp[4]{}(\iwF{4})}^2 }{n}\\
		&\qquad\qquad+ 52 \cstC[\denU] \cstV[\denX|\denU]  L_{\gamma/2} \frac{\Vert \mathds{1}_{[-k,k]} / \Mcg{c}{\denU} \Vert_{\Lp[\infty]{}(\wf)}^2 }{n}\\
		&= \tau^{U_k} (\gamma)
	\end{align*}
	which shows \cref{eq::quantile-ineq} and completes the proof. 
\end{proof}

\subsubsection{\textit{Linear statistic result}}\label{sec::lin-stat-results}
Consider the centred linear statistic $W_k$ defined in \cref{eq::wk-app} admitting as kernel the real-valued function for $y\in\pRz$
\begin{align*}
	h_k(y):= \int_{-k}^k  \vert\Mcg{c}{\denU}(t)\vert^{-2}  \cfun{y}{t} (\overline{\Mcg{c}{\denY}}(t) - \overline{\Mcg{c}{\denNY}}(t))  \iwF{2}(t) d\Lm(t).
\end{align*}
Using Markov's inequality, we show the following result for its quantile-bound.
\begin{lem}[Quantile of the linear statistic]\label{lem::linear-stat-quantile}
	Let \cref{ass:well-definedness-testing,ass::error-well-defined,ass::error-p-moments} be satisfied. Further, assume that $\denX\in\Lp[1]{+}(\basMSy{2(c-1)})$. Consider for $n\in\Nz$ the linear statistic $W_k$ defined in \cref{eq::wk-app}. For $\gamma\in(0,1)$ setting
	\begin{align*}
		\tau^{W_k}(\gamma) := \frac{1}{4} \Vert (\Mcg{c}{\denY}-\Mcg{c}{\denNY}) \mathds{1}_{[-k,k]} / \Mcg{c}{\denU} \Vert_{\Lp[2]{}(\iwF{2})}^2 + \frac{\cstC[\denU] \cstV[\denX|\denU]\Vert \mathds{1}_{[-k,k]} / \Mcg{c}{\denU} \Vert_{\Lp[\infty]{}(\wf)}^2 }{\gamma n}
	\end{align*}
	we have
	\begin{align*}
		\ipM[\denU,\denX]{n}(W_k \geq \tau^{W_k}(\gamma) ) \leq \gamma.
	\end{align*}
\end{lem}

\begin{proof}[Proof of \cref{lem::linear-stat-quantile}]
	By exploiting that the observations $(Y_j)_{j\in\nset{n}}$ are i.i.d. we get that 
	\begin{align*}
		n \iEx[\denY]{n} [ \vert W_k \vert^2] &= \Ex[\denY][\vert h_k(Y_1)\vert^2]\\
		&= \Ex[\denY]\left[ \left\vert \int_{-k}^k \frac{1}{ \vert \Mcg{c}{\denU}(t)\vert^2}\cfun{Y_1}{t} (\overline{\Mcg{c}{\denY}}(t) - \overline{\Mcg{c}{\denNY}}(t)) \iwF{2}(t) d\Lm(t)\right\vert^2\right]\\
		&\leq \Ex[\denY]\left[ \left\vert \int_{-k}^k \frac{1}{ \vert \Mcg{c}{\denU}(t)\vert^2}Y_1^{c-1+2\pi it} (\overline{\Mcg{c}{\denY}}(t) - \overline{\Mcg{c}{\denNY}}(t)) \iwF{2}(t) d\Lm(t)\right\vert^2\right]\\
		\intertext{With \cref{lem::Mellinprops} (ii) and $\Vert \denY\Vert_{\Lp[\infty]{+}(\basMSy{2c-1})}\leq \cstV[\denX|\denU]   \cstC[\denU]$ we obtain}
		&\Ex[\denY]\left[ \left\vert \int_{-k}^k \frac{1}{ \vert \Mcg{c}{\denU}(t)\vert^2}Y_1^{c-1+2\pi it} (\overline{\Mcg{c}{\denY}}(t) - \overline{\Mcg{c}{\denNY}}(t)) \iwF{2}(t) d\Lm(t)\right\vert^2\right]\\
		&\leq \cstC[\denU] \cstV[\denX|\denU] \Vert (\Mcg{c}{\denY}-\Mcg{c}{\denNY})\mathds{1}_{[-k,k]} /\Mcg{c}{\denU} \Vert_{\Lp[2]{}(\iwF{4})}^2.
	\end{align*}
	Exploiting $ab \leq a^2/4 + b^2$ for $a,b\in\Rz$, we get
	\begin{align*}
		&\gamma^{-1/2} (\cstC[\denU] \cstV[\denX|\denU])^{1/2} \Vert (\Mcg{c}{\denY}-\Mcg{c}{\denNY})\mathds{1}_{[-k,k]} /\Mcg{c}{\denU} \Vert_{\Lp[2]{}(\iwF{4})} n^{-1/2}\\
		&\leq \gamma^{-1/2} (\cstC[\denU] \cstV[\denX|\denU])^{1/2} \Vert \mathds{1}_{[-k,k]} /\Mcg{c}{\denU}  \Vert_{\Lp[\infty]{}(\wf)} \Vert (\Mcg{c}{\denY}-\Mcg{c}{\denNY})\mathds{1}_{[-k,k]} /\Mcg{c}{\denU} \Vert_{\Lp[2]{}(\iwF{2})} n^{-1/2}\\
		&\leq \frac{1}{4}\Vert (\Mcg{c}{\denY}-\Mcg{c}{\denNY})\mathds{1}_{[-k,k]} /\Mcg{c}{\denU} \Vert_{\Lp[2]{}(\iwF{2})}^2 + \gamma^{-1}n^{-1} \cstC[\denU] \cstV[\denX|\denU] \Vert \mathds{1}_{[-k,k]} /\Mcg{c}{\denU}  \Vert_{\Lp[\infty]{}(\wf)}^2 \\
		&= \tau^{W_k}(\gamma)
	\end{align*}
	and, consequently,
	\begin{align*}
		n \iEx[\denY]{n} [ \vert W_k \vert^2] \leq n\gamma \vert \tau^{W_k}(\gamma)\vert^2.
	\end{align*}
	Finally, by applying Markov's inequality follows immediately the result, that is,
	\begin{align*}
		\ipM[\denU,\denX]{n} (W_k \geq \tau^{W_k}(\gamma) ) \leq \frac{\iEx[\denY]{n} [ \vert W_k \vert^2] }{\vert \tau^{W_k}(\gamma)\vert^2}  \leq \gamma
	\end{align*}
	which completes the proof.
\end{proof}

\end{document}